\documentclass[11pt, reqno]{amsart}
\usepackage{subfig, enumerate, tikz-cd}
\usepackage{bookmark}
\usepackage{filecontents}
\usepackage[shortlabels]{enumitem}
\usepackage[section] {placeins}
\usepackage{cases}
\usepackage{amscd,amsfonts,amssymb,amsmath,tikz-cd}
\usepackage{bm}
\usepackage{hyperref}
\usepackage{epsfig}
\usepackage{mathtools}
\usepackage{float}
\usepackage{color}
\usepackage[normalem]{ulem}
\usepackage[sort]{cite}
\newtheorem{theorem}{Theorem}[section]

\newtheorem{proposition}[theorem]{Proposition}
\newtheorem{corollary}[theorem]{Corollary}

\newtheorem{lemma}[theorem]{Lemma}

\newtheorem{example}[theorem]{Example}

\theoremstyle{definition}

\newcommand{\Aut}{\operatorname{Aut}}

\newcommand{\id}{\mathrm{id}}

\newcommand{\cl}{\mathrm{cl}}

\newcommand{\overbar}[1]{\mkern 1.5mu\overline{\mkern-1.5mu#1\mkern-1.5mu}\mkern 1.5mu}

\usepackage[autostyle]{csquotes}

\begin{document}
\title{Bridging colorings of virtual links from virtual biquandles to biquandles}
\author{Mohamed Elhamdadi}
\author{Manpreet Singh}
\address{University of South Florida, \\
Tampa, FL, 33620}
\email{emohamed@usf.edu}
\email{manpreet.singh2@fulbrightmail.org\\manpreetsingh@usf.edu}

\subjclass[2020]{57K12, 57K10}
\keywords{Virtual knots, virtual braids, biquandles, virtual biquandles}

\begin{abstract}
A biquandle is a solution to the set-theoretical Yang-Baxter equation, which yields invariants for virtual knots such as the coloring number and the state-sum invariant. A virtual biquandle enriches the structure of a biquandle by incorporating an invertible unary map. This unary operator plays a crucial role in defining the action of virtual crossings on the labels of incoming arcs in a virtual link diagram. This leads to extensions of invariants from biquandles to virtual biquandles, thereby enhancing their strength.

In this article, we establish a connection between the coloring invariant derived from biquandles and virtual biquandles. We prove that the number of colorings of a virtual link $L$ by virtual biquandles can be recovered from colorings by biquandles. We achieve this by proving the equivalence between two different representations of virtual braid groups. Furthermore, we introduce a new set of labeling rules using which one can construct a presentation of the associated fundamental virtual biquandle of $L$ using only the relations coming from the classical crossings. This is an improvement to the traditional method, where writing down a presentation of the associated fundamental virtual biquandle necessitates noting down the relations arising from the classical and virtual crossings.
\end{abstract}

\maketitle
\section{Introduction}
In classical knot theory, the double points in a knot diagram indicate the over and under crossing information. Kauffman \cite{MR1721925} generalized these diagrams by introducing a new type of crossing which neither indicates which arc passes over nor which arc passes under. These crossing are termed as virtual crossings and are indicated as \begin{tabular}{@{}c@{}}\includegraphics[width=3ex]{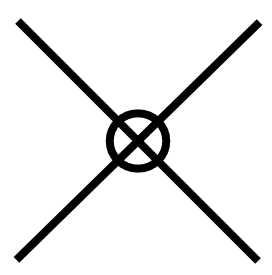}\end{tabular}. Diagrams with such crossings are referred as virtual link diagrams. These diagrams are considered up to the equivalence relations generated by planar isotopies and generalized Reidemeister moves. These moves encompass both traditional Reidemeister moves and their analogues involving virtual crossings. In \cite{MR1905687, MR1749502}, virtual knot theory is interpreted as knot theory in oriented thickened surfaces up to a certain equivalence. Subsequently, Kuperberg \cite{MR1997331} proved that virtual knot theory is indeed the study of links in $3$-manifolds, thereby deducing that virtual knot theory represents a true topological extension of classical knot theory.

Kauffman \cite{MR1721925} also introduced analogues of Artin braid groups, referred to as virtual braid groups. Geometrically, in addition to classical generators, virtual braid groups have generators that involve virtual crossings and an additional set of relations. These added relations encode new moves beyond the classical Reidemeister moves. Moreover, Kamada \cite{MR2351010} and Kauffman-Lambropoulou \cite{MR2371718, MR2128049} proved Alexander and Markov theorems for oriented virtual knots. Since the introduction of virtual knot theory, numerous works have focused on virtual braids \cite{MR2493369, MR3597250, MR4352636} and virtual links \cite{MR1721925, MR1914297, MR2100692, MR2153118, MR2515811, MR3717970}.  

The notions of birack and biquandle, introduced in  \cite{MR2100870}, yield natural invariants of virtual braids and virtual knots, respectively. In \cite{MR3666513}, computations of biracks and biquandles of small sizes are presented, and, as an application, it is reproved that the Burau map from braids to matrices is not injective. In \cite{MR2191942}, the introduction of virtual quandles and virtual biquandles  extends the structure of a quandle and a biquandle, respectively, by incorporating an invertible unary map. This unary operator plays a crucial role in defining the action of virtual crossings on the labels of incoming arcs in a virtual knot diagram. Leveraging this unary operator, numerous invariants originating from biquandles are extended to virtual biquandles, thereby enhancing their strength. For instance, a notion of virtual knot group is defined \cite{MR1916950}, whose various generalizations are explored in \cite{MR3640614}. Furthermore, state-sum invariants are defined for virtual knots using virtual biquandles in \cite{MR3982044, 
MR2515811}.

 In this article, we found a connection between the coloring number invariant derived from biquandles and virtual biquandles. We prove that, for a given virtual link $L$ and a virtual biquandle $(X,f,R)$, the set of colorings of $L$ by $(X,f,R)$ is in bijection with the set of colorings of $L$ by a biquandle $(X,VR)$, where $VR$ is a new operator defined on $X$ (see Theorem \ref{thm:number_of_colorings_are_same}) and defines a functor from the category of virtual biquandles to the category of biquandles. As a result, in Corollary \ref{cor:cor_number_of_colorings_are_same}, we prove that the number of colorings of $L$ by virtual biquandles can be recovered from colorings by biquandles. Furthermore, by introducing a new set of labeling rules (see Figure \ref{fig:new_labeling_virtual_crossing_rules}), we construct a presentation of the associated fundamental virtual biquandle of $L$ using only the relations coming from the classical crossings (see Theorem \ref{thm:virtual_biquandle_isomorphism}). This is a notable improvement over the conventional approach, which typically involves recording relations from both classical and virtual crossings. Thus, with the adoption of these new labeling rules, one can directly extract the fundamental virtual biquandle of $L$ from its Gauss diagram (using rules depicted in Figure \ref{fig:depiction_of_labeling_rules_on_Gauss_diagrms}), which was not the case before.

This article is organized as follows.  In Section~\ref{pre} we revisit the fundamentals of virtual knot theory. Specifically, we review the concept of equivalence of virtual link diagrams,  the definition of the fundamental virtual biquandle of a virtual link, and the colorings of virtual links using virtual biquandles. Additionally, we note that for a given virtual biquandle $(X,f,R)$, there exists an operator $VR$ distinct from $R$, for which $(X,VR)$ forms a biquandle.

In Section~\ref{Col}, we recall the definition of the virtual braid group $VB_n$ and define two representations of $VB_n$ to the automorphism group $\Aut(X^n)$  of a virtual biquandle $(X, f, R)$. In Theorem \ref{thm:number_of_colorings_are_same}, we prove that if $L$ is a virtual link, then the set of colorings of $L$ by a virtual biquandle $(X,f,R)$, following the rules illustrated in Figure \ref{fig:old_virtual_crossing_rules}, is in bijection with the set of colorings of $L$ by $(X,VR)$ under the rules depicted in Figure \ref{fig:biquandle_crossing_rules}. In particular, we show that the coloring number of a virtual link  $L$ by virtual biquandles can be retrieved from colorings by biquandles.

In Section~\ref{Relation}, we present two representations of the virtual braid group $VB_n$ to the automorphism group of free virtual biquandle on a set of $n$ elements. By employing  these representations, we prove that for a given virtual link $L$, the fundamental virtual biquandle $(VBQ(L), f, R)$ is isomorphic to $(VBQ(L), f, VR)$ as virtual biquandles (refer to Theorem \ref{thm:virtual_biquandle_isomorphism}). This implies that we can construct a presentation of the fundamental virtual biquandle of $L$ simply by recording the relations arising at classical crossing, using the new labeling rules shown in Figure \ref{fig:new_labeling_virtual_crossing_rules}. Consequently, $(VBQ(L), f, R)$ can be recovered from a Gauss diagram representing $L$, using the rules shown in Figure \ref{fig:depiction_of_labeling_rules_on_Gauss_diagrms}.

In Section \ref{sec:conclusion}, we discuss potential directions stemming from this work.

\section{Preliminaries}\label{pre}
A non-empty set $X$ with an operator $R\colon X \times X \to X \times X$, mapping $(x,y) \mapsto (R_1(x,y), R_2(x,y) )$, is called a {\it biquandle} if the following holds:
\begin{enumerate}
\item $R$ satisfies the Yang-Baxter relation, that is,
\begin{align}
(R \times \id) (\id \times R) (R \times \id)=(\id \times R) (R \times \id) (\id \times R),\label{eq:YBE}
\end{align}
where $\id$ is the identity map on $X$;
\item $R$ is invertible;
\item $R_1$ is {\it left-invertible}, meaning that for any $x,z \in X$, there exists a unique $y \in X$ such that $R_1(x,y)=z$;
\item $R_2$ is {\it right-invertible}, meaning that for any $y, w \in X$, there exists a unique $x \in X$ such that $R_2(x,y)=w$;
\item $R$ satisfies the {\it type I condition}, which is, for any $a \in X$, there exists a unique $x \in X$ such that $R(x,a)=(x,a)$.
\end{enumerate}

Note that the type I condition is equivalent to the following: for given $a \in X$, there exists a unique $y \in X$ such that $R(a, y)=(a,y)$.

The following are some examples.
\begin{example}\label{Wada}{\rm
	Let $G$ be a group and let $R\colon G \times G \to G \times G$ be a map sending $(x,y)$ to $(x^{-1}y^{-1}x, y^2x )$, then $(G,R)$ is a biquandle called {\it Wada biquandle}.  For more examples of biquandles arising from groups, refer to \cite{MR1167178}.}
\end{example}

\begin{example} \label{YBEexamples} {\rm 
		Let $k$ be a commutative ring with unity $1$, and units $\alpha$ and $\beta$, such that $(1-\alpha)(1-\beta)=0$. 
		Then   $\displaystyle R= \left[ \begin{array}{cc} 1-\alpha & \alpha \\ \beta & 1-\beta \end{array} \right] $ makes the ring $k$ into a biquandle.   This example can be generalized by replacing $k$ with $k^m$.  Precisely, let $R\colon k^m \times k^m \rightarrow  k^m \times k^m $ be given by a matrix 
		$\displaystyle R= \left[ \begin{array}{cc} I - A & A \\ B & I-B \end{array} \right], $
		where the matrices $A$ and $B$ are invertible $m \times m$ matrices with elements in $k$, satisfying $AB=BA$ and $(I-A)(I-B)=0$, where $I$ is the identity matrix.  Then  $(k^m,R)$ is a biquandle. 
		
} \end{example}

We revisit the definition of virtual biquandles from \cite{MR2191942}. A {\it virtual biquandle} is a triplet $(X,f,R)$, where $(X,R)$ is a biquandle and $f: X \to X$ is an automorphism of $(X,R)$.

The following are straightforward examples of virtual biquandles.
\begin{example}\label{Vwada}
{\rm	Consider the Wada biquandle  $(G,R)$ as described in Example \ref{Wada}.  If $f$ is a group automorphism $f:G \to G$, then the $(G,f,R)$ forms a virtual biquandle.}
\end{example}

\begin{example}\label{Vlinear}
{\rm	Consider the biquandle $(k,R)$ defined in Example \ref{YBEexamples}. If $f$ is an automorphism of the ring $k$, then $(k,f,R)$ forms a virtual biquandle.
}
\end{example}

A {\it homomorphism} from $(X,f,R)$ to $(Y,f',R')$ is a map $\phi: (X,f,R) \to (Y,f', R')$ such that $\phi:(X,R) \to (Y, R')$ is a biquandle homomorphism and $\phi \circ f= f' \circ \phi$.

The proof of Proposition \ref{prop:twisted_biquandle} can be found in \cite{2024arXiv240101533E}; however, for the sake of completeness, we recall it here.

\begin{proposition}\label{prop:twisted_biquandle}
If $(X,f,R)$ is a virtual biquandle, then $(X,VR)$ is a biquandle, where $VR\colon X \times X \to X \times X$ defined as $VR(x,y)=(R_1(x,f(y)), R_2(f^{-1}(x),y))$. Moreover, $f$ is also an automorphism of $(X,VR)$.
\end{proposition}
\begin{proof}
It is sufficient to prove that $VR$ is a Yang-Baxter operator.

L.H.S of Equation \eqref{eq:YBE} is
\begin{align*}
(VR \times \id ) (\id \times VR) (VR \times \id)~~(x,y,z)=&(VR \times \id) (\id \times VR)~~ \big( R_1(x,f(y)), R_2(f^{-1}(x),y), z \big)\\
=&(VR \times \id ) ~~ \big( R_1(x,f(y)), R_1(  R_2 (f^{-1}(x), y), f(z) ),\\& R_2 ( R_2 (f^{-2}(x), f^{-1}(y), z) \big)\\
=&\big( R_1 \big( R_1(x,f(y), R_1(R_2(x,f(y)), f^2(z) ) \big), \\&R_2 \big( R_1 (f^{-1}(x), y), R_1 (R_2(f^{-1}(x),y), f(z)\big),\\& R_2 ( R_2 (f^{-2}(x), f^{-1}(y), z) \big).
\end{align*}

R.H.S of Equation \eqref{eq:YBE} is
\begin{align*}
(\id \times VR)
 (VR \times \id) 
(\id \times VR)~~(x,y,z)=&(\id \times VR) 
 (VR \times \id)
 ~~\big( x, R_1(y,f(z)), R_2 (f^{-1}(y),z) \big)\\
=&(\id \times VR) ~~R_1\big( x, R_1(f(y), f^2(z)), R_2(f^{-1}(x), R_1(y, f(z))),\\& R_2(f^{-1}(y),z) \big)\\
=& \big( R_1(x, R_1(f(y), f^2(z))), R_1(R_2(f^{-1}(x), R_1(y, f(z))), R_2(y, f(z)) \big),\\
&R_2(R_2(f^{-2}(x), R_1(f^{-1}(y),z)), R_2(f^{-1}(y),z)) \big).
\end{align*}
Now L.H.S=R.H.S, if and only if the following equations hold for all $x,y,z \in X$:
\begin{align}
R_1 \big( R_1(x,f(y), R_1(R_2(x,f(y)), f^2(z) ) \big)= R_1\big(x, R_1(f(y), f^2(z))\big) \label{eq:TR_III_1}\\
R_2 \big( R_1 (f^{-1}(x), y), R_1 (R_2(f^{-1}(x),y), f(z)\big)= R_1\big(R_2(f^{-1}(x), R_1(y, f(z))\big), R_2(y,f(z))\big) \label{eq:TR_III_2}\\
R_2 \big( R_2 (f^{-2}(x), f^{-1}(y), z\big) =R_2\big(R_2(f^{-2}(x), R_1(f^{-1}(y),z)), R_2(f^{-1}(y),z)\big). \label{eq:TR_III_3}
\end{align}
Equations \ref{eq:TR_III_1}, \ref{eq:TR_III_2} and \ref{eq:TR_III_3} hold because $R$ is a Yang-Baxter operator.
\end{proof}

It can easily be verified that the construction of a biquandle $(X, VR)$ from the virtual biquandle $(X,f,R)$ in Proposition \ref{prop:twisted_biquandle} is functorial.
Moreover, from Proposition \ref{prop:twisted_biquandle} it follows that $(X,f, VR)$ is a virtual biquandle.

We recall the following notions of virtual knot theory from \cite{MR1721925}.

A {\it virtual link diagram} is a generic immersion of a finitely many circles into the Euclidean plane $\mathbb{R}^2$ or the $2$-sphere $\mathbb{S}^2$, where the double points are decorated either as classical crossings or virtual crossings, as shown in Figure \ref{fig:virtual_crossing}.
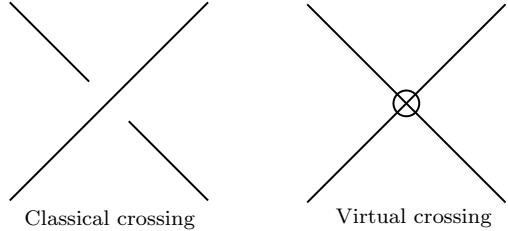
\begin{figure}[H]

\tikzset{every picture/.style={line width=0.75pt}} %set default line width to 0.75pt        

\begin{tikzpicture}[x=0.75pt,y=0.75pt,yscale=-1,xscale=1]

%Straight Lines [id:da2997457116729918] 
\draw    (450,51) -- (550,151) ;
%Straight Lines [id:da6212322998062861] 
\draw    (550,51) -- (450,151) ;
%Shape: Circle [id:dp05543145449426401] 
\draw   (493.5,101) .. controls (493.5,97.41) and (496.41,94.5) .. (500,94.5) .. controls (503.59,94.5) and (506.5,97.41) .. (506.5,101) .. controls (506.5,104.59) and (503.59,107.5) .. (500,107.5) .. controls (496.41,107.5) and (493.5,104.59) .. (493.5,101) -- cycle ;
%Straight Lines [id:da3365610143383986] 
\draw    (300,50) -- (340,90) ;
%Straight Lines [id:da30340848213721194] 
\draw    (360,110) -- (400,150) ;
%Straight Lines [id:da820341715987933] 
\draw    (300,150) -- (400,50) ;

% Text Node
\draw (306,153.4) node [anchor=north west][inner sep=0.75pt]  [font=\scriptsize]  {Classical\ crossing};
% Text Node
\draw (463,152.4) node [anchor=north west][inner sep=0.75pt]  [font=\scriptsize]  {Virtual\ crossing};

\end{tikzpicture}

\caption{Classical crossing and Virtual crossing.} \label{fig:virtual_crossing}
\end{figure}

Two virtual link diagrams are said to be {\it equivalent} if they are related by a finite sequence of planar isotopies and {\it generalized Reidemeister moves} shown in Figure \ref{fig:g_reidemeister_moves}. A {\it virtual link} is an equivalence of virtual link diagrams under  the generalized Reidemeister moves.

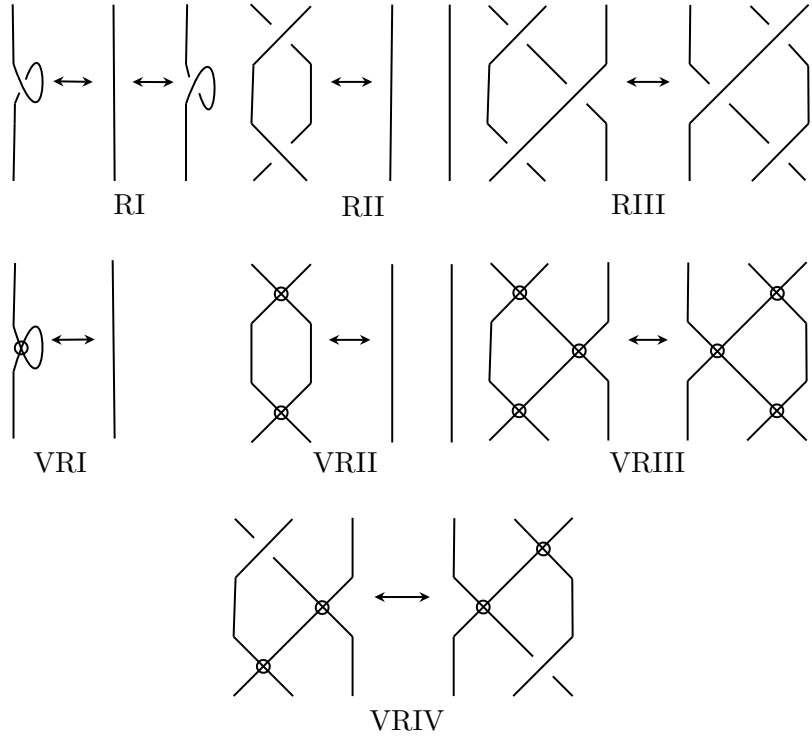
\begin{figure}[H]

\tikzset{every picture/.style={line width=0.75pt}} %set default line width to 0.75pt        

\begin{tikzpicture}[x=0.75pt,y=0.75pt,yscale=-1,xscale=1]
%uncomment if require: \path (0,447); %set diagram left start at 0, and has height of 447

%Straight Lines [id:da6454540190052511] 
\draw [line width=0.75]    (89.58,10.83) -- (90,41) ;
%Curve Lines [id:da4276564111198674] 
\draw [line width=0.75]    (90,41) .. controls (108.71,94.57) and (108.09,15.81) .. (96.14,48.86) ;
%Straight Lines [id:da03302345293255893] 
\draw [line width=0.75]    (90.71,60.86) -- (90.58,70.83) ;
%Straight Lines [id:da8084665041098035] 
\draw [line width=0.75]    (90.58,70.83) -- (90,100) ;
%Straight Lines [id:da4914002969213168] 
\draw [line width=0.75]    (140.43,11) -- (141,100) ;
%Straight Lines [id:da08224187743978084] 
\draw [line width=0.75]    (209.6,10.8) -- (220,21) ;
%Straight Lines [id:da18245176570993593] 
\draw [line width=0.75]    (240,11.8) -- (210.6,41.8) ;
%Straight Lines [id:da47922209067524557] 
\draw [line width=0.75]    (230,31) -- (240,41) ;
%Straight Lines [id:da6425907569874005] 
\draw [line width=0.75]    (240,71) -- (230,81) ;
%Straight Lines [id:da9257223243827232] 
\draw [line width=0.75]    (211,42) -- (210,71) ;
%Straight Lines [id:da49482119906566324] 
\draw [line width=0.75]    (238,99.9) -- (210,71) ;
%Straight Lines [id:da8742688214218626] 
\draw [line width=0.75]    (220,91) -- (211,99.9) ;
%Straight Lines [id:da2965017742609686] 
\draw [line width=0.75]    (281,11) -- (280,100) ;
%Straight Lines [id:da8528729765664949] 
\draw [line width=0.75]    (329.67,11.5) -- (339.33,21.17) ;
%Straight Lines [id:da5965684349542517] 
\draw [line width=0.75]    (358.67,11.67) -- (330,41.17) ;
%Straight Lines [id:da6758199151142151] 
\draw [line width=0.75]    (389,11) -- (389,41) ;
%Straight Lines [id:da9908423738429283] 
\draw [line width=0.75]    (349,31) -- (369,51) ;
%Straight Lines [id:da9770146977190732] 
\draw [line width=0.75]    (389,41) -- (330,100.1) ;
%Straight Lines [id:da7506543692071446] 
\draw [line width=0.75]    (379,61) -- (389,71) ;
%Straight Lines [id:da4592370921246929] 
\draw [line width=0.75]    (389,71) -- (389,99.9) -- (389,100) ;
%Straight Lines [id:da9586715904686842] 
\draw [line width=0.75]    (330,41.17) -- (329,71) ;
%Straight Lines [id:da43205972937501913] 
\draw [line width=0.75]    (349,91) -- (358.2,100.1) ;
%Straight Lines [id:da6751496609273613] 
\draw [line width=0.75]    (329,71) -- (339,81) ;
%Straight Lines [id:da916957741554555] 
\draw [line width=0.75]    (431.33,11.17) -- (431,41) ;
%Straight Lines [id:da009681286680223056] 
\draw [line width=0.75]    (481,31) -- (490.8,41.8) ;
%Straight Lines [id:da1284127823989475] 
\draw [line width=0.75]    (491,11) -- (431,71) ;
%Straight Lines [id:da6299980686574169] 
\draw [line width=0.75]    (431,71) -- (431,100) ;
%Straight Lines [id:da9795928731959931] 
\draw [line width=0.75]    (431,41) -- (441,51) ;
%Straight Lines [id:da491745640594614] 
\draw [line width=0.75]    (451,61) -- (471,81) ;
%Straight Lines [id:da8650391449459701] 
\draw [line width=0.75]    (490.8,41.8) -- (491,71) ;
%Straight Lines [id:da1722584594580765] 
\draw [line width=0.75]    (491,71) -- (462.4,100.1) ;
%Straight Lines [id:da6159775870393777] 
\draw [line width=0.75]    (481,91) -- (489.4,99.7) ;
%Straight Lines [id:da50133284727544] 
\draw [line width=0.75]    (461.67,11.67) -- (471,21) ;
%Straight Lines [id:da36535402888154944] 
\draw [line width=0.75]    (93,55.43) -- (90.71,60.86) ;
%Straight Lines [id:da13565730705529566] 
\draw [line width=0.75]    (177,100) -- (177,61) ;
%Curve Lines [id:da4989104555072874] 
\draw [line width=0.75]    (177,61) .. controls (196,8) and (194,88.67) .. (183.67,55.67) ;
%Straight Lines [id:da37031611301647427] 
\draw [line width=0.75]    (177,41) -- (177.58,10.83) ;
%Straight Lines [id:da11517751379326424] 
\draw [line width=0.75]    (177,41) ;
%Straight Lines [id:da03144984945805562] 
\draw [line width=0.75]    (177,41) -- (178.67,47.67) ;
%Straight Lines [id:da9035923029043199] 
\draw [line width=0.75]    (113,49.72) -- (127,49.95) ;
\draw [shift={(130,50)}, rotate = 180.95] [fill={rgb, 255:red, 0; green, 0; blue, 0 }  ][line width=0.08]  [draw opacity=0] (5.36,-2.57) -- (0,0) -- (5.36,2.57) -- (3.56,0) -- cycle    ;
\draw [shift={(110,49.67)}, rotate = 0.95] [fill={rgb, 255:red, 0; green, 0; blue, 0 }  ][line width=0.08]  [draw opacity=0] (5.36,-2.57) -- (0,0) -- (5.36,2.57) -- (3.56,0) -- cycle    ;
%Straight Lines [id:da8163500690679901] 
\draw [line width=0.75]    (153,50.02) -- (167.75,50.11) ;
\draw [shift={(170.75,50.13)}, rotate = 180.35] [fill={rgb, 255:red, 0; green, 0; blue, 0 }  ][line width=0.08]  [draw opacity=0] (5.36,-2.57) -- (0,0) -- (5.36,2.57) -- (3.56,0) -- cycle    ;
\draw [shift={(150,50)}, rotate = 0.35] [fill={rgb, 255:red, 0; green, 0; blue, 0 }  ][line width=0.08]  [draw opacity=0] (5.36,-2.57) -- (0,0) -- (5.36,2.57) -- (3.56,0) -- cycle    ;
%Straight Lines [id:da3011873115226572] 
\draw [line width=0.75]    (253,50) -- (268,50) ;
\draw [shift={(271,50)}, rotate = 180] [fill={rgb, 255:red, 0; green, 0; blue, 0 }  ][line width=0.08]  [draw opacity=0] (5.36,-2.57) -- (0,0) -- (5.36,2.57) -- (3.56,0) -- cycle    ;
\draw [shift={(250,50)}, rotate = 0] [fill={rgb, 255:red, 0; green, 0; blue, 0 }  ][line width=0.08]  [draw opacity=0] (5.36,-2.57) -- (0,0) -- (5.36,2.57) -- (3.56,0) -- cycle    ;
%Straight Lines [id:da22868187135810947] 
\draw [line width=0.75]    (402,50) -- (418,50) ;
\draw [shift={(421,50)}, rotate = 180] [fill={rgb, 255:red, 0; green, 0; blue, 0 }  ][line width=0.08]  [draw opacity=0] (5.36,-2.57) -- (0,0) -- (5.36,2.57) -- (3.56,0) -- cycle    ;
\draw [shift={(399,50)}, rotate = 0] [fill={rgb, 255:red, 0; green, 0; blue, 0 }  ][line width=0.08]  [draw opacity=0] (5.36,-2.57) -- (0,0) -- (5.36,2.57) -- (3.56,0) -- cycle    ;
%Straight Lines [id:da25755142783676843] 
\draw [line width=0.75]    (310,11) -- (310,100) ;
%Straight Lines [id:da45017751191485844] 
\draw [line width=0.75]    (240,41) -- (240,71) ;
%Straight Lines [id:da22950302264560873] 
\draw [line width=0.75]    (90.8,141.4) -- (90,173.25) ;
%Curve Lines [id:da003341829608865643] 
\draw [line width=0.75]    (90,173.25) .. controls (109.5,237.75) and (109.5,129.25) .. (90,196.25) ;
%Straight Lines [id:da885587891024255] 
\draw [line width=0.75]    (90,196.25) -- (90,230) ;
%Straight Lines [id:da6881791368661896] 
\draw [line width=0.75]    (140,140) -- (141,230) ;
%Straight Lines [id:da6246310441307547] 
\draw [line width=0.75]    (210,142) -- (240,172) ;
%Straight Lines [id:da9110491779937081] 
\draw [line width=0.75]    (240,172) -- (240,202) ;
%Straight Lines [id:da3155794738361539] 
\draw [line width=0.75]    (240,202) -- (210,232) ;
%Straight Lines [id:da07224943325032396] 
\draw [line width=0.75]    (240,142) -- (210,172) ;
%Straight Lines [id:da2428588300420781] 
\draw [line width=0.75]    (210,172) -- (210,202) ;
%Straight Lines [id:da49715752187735185] 
\draw [line width=0.75]    (210,202) -- (240,232) ;
%Straight Lines [id:da7391784516831209] 
\draw [line width=0.75]    (281,142) -- (281,232) ;
%Straight Lines [id:da675194581655465] 
\draw [line width=0.75]    (311,142) -- (311,232) ;
%Straight Lines [id:da2825952336344899] 
\draw [line width=0.75]    (330.67,141.5) -- (390,201) ;
%Straight Lines [id:da8812759461032905] 
\draw [line width=0.75]    (359.67,141.67) -- (331,171.17) ;
%Straight Lines [id:da2060979480197429] 
\draw [line width=0.75]    (390,171) -- (330,231) ;
%Straight Lines [id:da8955585704043866] 
\draw [line width=0.75]    (390,201) -- (390,231) ;
%Straight Lines [id:da5600059378779946] 
\draw [line width=0.75]    (331,171.17) -- (330,201) ;
%Straight Lines [id:da8862827074774909] 
\draw [line width=0.75]    (330,201) -- (360,231) ;
%Straight Lines [id:da03596062782092668] 
\draw [line width=0.75]    (430.33,141.17) -- (430,171) ;
%Straight Lines [id:da32944772716725357] 
\draw [line width=0.75]    (490,141) -- (430,201) ;
%Straight Lines [id:da4101474982720297] 
\draw [line width=0.75]    (430,201) -- (430,231) ;
%Straight Lines [id:da9184359435670643] 
\draw [line width=0.75]    (489.8,171.8) -- (490,201) ;
%Straight Lines [id:da6121084633357464] 
\draw [line width=0.75]    (490,201) -- (460,231) ;
%Straight Lines [id:da8064431777001504] 
\draw [line width=0.75]    (430,171) -- (490,231) ;
%Straight Lines [id:da14854095204946782] 
\draw [line width=0.75]    (460.67,141.67) -- (489.8,171.8) ;
%Straight Lines [id:da9082391181324421] 
\draw [line width=0.75]    (403,180.2) -- (417,180.2) ;
\draw [shift={(420,180.2)}, rotate = 180] [fill={rgb, 255:red, 0; green, 0; blue, 0 }  ][line width=0.08]  [draw opacity=0] (5.36,-2.57) -- (0,0) -- (5.36,2.57) -- (3.56,0) -- cycle    ;
\draw [shift={(400,180.2)}, rotate = 0] [fill={rgb, 255:red, 0; green, 0; blue, 0 }  ][line width=0.08]  [draw opacity=0] (5.36,-2.57) -- (0,0) -- (5.36,2.57) -- (3.56,0) -- cycle    ;
%Shape: Circle [id:dp9512407878037892] 
\draw  [line width=0.75]  (92.93,181.16) .. controls (94.64,180.62) and (96.45,181.57) .. (96.99,183.28) .. controls (97.52,184.99) and (96.57,186.81) .. (94.86,187.34) .. controls (93.16,187.87) and (91.34,186.92) .. (90.8,185.22) .. controls (90.27,183.51) and (91.22,181.69) .. (92.93,181.16) -- cycle ;
%Shape: Circle [id:dp23187988814443328] 
\draw  [line width=0.75]  (224.03,153.91) .. controls (225.74,153.37) and (227.56,154.32) .. (228.09,156.03) .. controls (228.63,157.74) and (227.68,159.56) .. (225.97,160.09) .. controls (224.26,160.63) and (222.44,159.68) .. (221.91,157.97) .. controls (221.37,156.26) and (222.32,154.44) .. (224.03,153.91) -- cycle ;
%Shape: Circle [id:dp17929235690606982] 
\draw  [line width=0.75]  (224.03,213.91) .. controls (225.74,213.37) and (227.56,214.32) .. (228.09,216.03) .. controls (228.63,217.74) and (227.68,219.56) .. (225.97,220.09) .. controls (224.26,220.63) and (222.44,219.68) .. (221.91,217.97) .. controls (221.37,216.26) and (222.32,214.44) .. (224.03,213.91) -- cycle ;
%Shape: Circle [id:dp46248970168232106] 
\draw  [line width=0.75]  (344.37,153.32) .. controls (346.07,152.79) and (347.89,153.74) .. (348.43,155.45) .. controls (348.96,157.16) and (348.01,158.97) .. (346.3,159.51) .. controls (344.59,160.04) and (342.78,159.09) .. (342.24,157.38) .. controls (341.71,155.68) and (342.66,153.86) .. (344.37,153.32) -- cycle ;
%Shape: Circle [id:dp041082590627170235] 
\draw  [line width=0.75]  (374.37,182.66) .. controls (376.07,182.12) and (377.89,183.07) .. (378.43,184.78) .. controls (378.96,186.49) and (378.01,188.31) .. (376.3,188.84) .. controls (374.59,189.38) and (372.78,188.43) .. (372.24,186.72) .. controls (371.71,185.01) and (372.66,183.19) .. (374.37,182.66) -- cycle ;
%Shape: Circle [id:dp8782577435597972] 
\draw  [line width=0.75]  (344.03,212.91) .. controls (345.74,212.37) and (347.56,213.32) .. (348.09,215.03) .. controls (348.63,216.74) and (347.68,218.56) .. (345.97,219.09) .. controls (344.26,219.63) and (342.44,218.68) .. (341.91,216.97) .. controls (341.37,215.26) and (342.32,213.44) .. (344.03,212.91) -- cycle ;
%Shape: Circle [id:dp5875257158942833] 
\draw  [line width=0.75]  (444.03,182.66) .. controls (445.74,182.12) and (447.56,183.07) .. (448.09,184.78) .. controls (448.63,186.49) and (447.68,188.31) .. (445.97,188.84) .. controls (444.26,189.38) and (442.44,188.43) .. (441.91,186.72) .. controls (441.37,185.01) and (442.32,183.19) .. (444.03,182.66) -- cycle ;
%Shape: Circle [id:dp7649813940153262] 
\draw  [line width=0.75]  (474.27,153.64) .. controls (475.97,153.11) and (477.79,154.06) .. (478.33,155.77) .. controls (478.86,157.47) and (477.91,159.29) .. (476.2,159.83) .. controls (474.49,160.36) and (472.68,159.41) .. (472.14,157.7) .. controls (471.61,155.99) and (472.56,154.18) .. (474.27,153.64) -- cycle ;
%Shape: Circle [id:dp21228994060833095] 
\draw  [line width=0.75]  (474.03,212.91) .. controls (475.74,212.37) and (477.56,213.32) .. (478.09,215.03) .. controls (478.63,216.74) and (477.68,218.56) .. (475.97,219.09) .. controls (474.26,219.63) and (472.44,218.68) .. (471.91,216.97) .. controls (471.37,215.26) and (472.32,213.44) .. (474.03,212.91) -- cycle ;
%Straight Lines [id:da5295828315702733] 
\draw [line width=0.75]    (390,141) -- (390,171) ;
%Straight Lines [id:da28454423906099546] 
\draw [line width=0.75]    (252,180.2) -- (267,180.2) ;
\draw [shift={(270,180.2)}, rotate = 180] [fill={rgb, 255:red, 0; green, 0; blue, 0 }  ][line width=0.08]  [draw opacity=0] (5.36,-2.57) -- (0,0) -- (5.36,2.57) -- (3.56,0) -- cycle    ;
\draw [shift={(249,180.2)}, rotate = 0] [fill={rgb, 255:red, 0; green, 0; blue, 0 }  ][line width=0.08]  [draw opacity=0] (5.36,-2.57) -- (0,0) -- (5.36,2.57) -- (3.56,0) -- cycle    ;
%Straight Lines [id:da5242531854609827] 
\draw [line width=0.75]    (112,180.17) -- (128,180.03) ;
\draw [shift={(131,180)}, rotate = 179.48] [fill={rgb, 255:red, 0; green, 0; blue, 0 }  ][line width=0.08]  [draw opacity=0] (5.36,-2.57) -- (0,0) -- (5.36,2.57) -- (3.56,0) -- cycle    ;
\draw [shift={(109,180.2)}, rotate = 359.48] [fill={rgb, 255:red, 0; green, 0; blue, 0 }  ][line width=0.08]  [draw opacity=0] (5.36,-2.57) -- (0,0) -- (5.36,2.57) -- (3.56,0) -- cycle    ;
%Straight Lines [id:da2478944274064392] 
\draw    (201.67,270.5) -- (211.33,280.17) ;
%Straight Lines [id:da36816334196954437] 
\draw    (230.67,270.67) -- (202,300.17) ;
%Straight Lines [id:da9508242119687754] 
\draw    (261,270) -- (261,300) ;
%Straight Lines [id:da5288519379697976] 
\draw    (221,290) -- (261,330) ;
%Straight Lines [id:da04627795948701652] 
\draw    (261,300) -- (201,360) ;
%Straight Lines [id:da880388409637783] 
\draw    (261,330) -- (261,360) ;
%Straight Lines [id:da41121069024529555] 
\draw    (202,300.17) -- (201,330) ;
%Straight Lines [id:da33428648039511333] 
\draw    (201,330) -- (231,360) ;
%Straight Lines [id:da23037002566403164] 
\draw    (312.33,270.17) -- (312,300) ;
%Straight Lines [id:da14221994034962926] 
\draw    (372,270) -- (312,330) ;
%Straight Lines [id:da6330786491080014] 
\draw    (312,330) -- (312,360) ;
%Straight Lines [id:da6997708637046459] 
\draw    (312,300) -- (352,340) ;
%Straight Lines [id:da4203875328775868] 
\draw    (371.8,300.8) -- (372,330) ;
%Straight Lines [id:da2257073818557721] 
\draw    (372,330) -- (342,360) ;
%Straight Lines [id:da09338364837027457] 
\draw    (362,350) -- (372,360) ;
%Straight Lines [id:da11640103549199687] 
\draw    (342.67,270.67) -- (371.8,300.8) ;
%Straight Lines [id:da33218740709167505] 
\draw    (275,310) -- (283.33,310) -- (297,310) ;
\draw [shift={(300,310)}, rotate = 180] [fill={rgb, 255:red, 0; green, 0; blue, 0 }  ][line width=0.08]  [draw opacity=0] (5.36,-2.57) -- (0,0) -- (5.36,2.57) -- (3.56,0) -- cycle    ;
\draw [shift={(272,310)}, rotate = 0] [fill={rgb, 255:red, 0; green, 0; blue, 0 }  ][line width=0.08]  [draw opacity=0] (5.36,-2.57) -- (0,0) -- (5.36,2.57) -- (3.56,0) -- cycle    ;
%Shape: Circle [id:dp6926069891848706] 
\draw   (356.27,282.64) .. controls (357.97,282.11) and (359.79,283.06) .. (360.33,284.77) .. controls (360.86,286.47) and (359.91,288.29) .. (358.2,288.83) .. controls (356.49,289.36) and (354.68,288.41) .. (354.14,286.7) .. controls (353.61,284.99) and (354.56,283.18) .. (356.27,282.64) -- cycle ;
%Shape: Circle [id:dp29637684191438307] 
\draw   (325.93,311.97) .. controls (327.64,311.44) and (329.46,312.39) .. (329.99,314.1) .. controls (330.53,315.81) and (329.58,317.62) .. (327.87,318.16) .. controls (326.16,318.69) and (324.34,317.74) .. (323.81,316.03) .. controls (323.27,314.33) and (324.22,312.51) .. (325.93,311.97) -- cycle ;
%Shape: Circle [id:dp06855633584527598] 
\draw   (215.03,341.91) .. controls (216.74,341.37) and (218.56,342.32) .. (219.09,344.03) .. controls (219.63,345.74) and (218.68,347.56) .. (216.97,348.09) .. controls (215.26,348.63) and (213.44,347.68) .. (212.91,345.97) .. controls (212.37,344.26) and (213.32,342.44) .. (215.03,341.91) -- cycle ;
%Shape: Circle [id:dp4373615321307919] 
\draw   (244.7,312.24) .. controls (246.41,311.71) and (248.22,312.66) .. (248.76,314.37) .. controls (249.29,316.07) and (248.34,317.89) .. (246.63,318.43) .. controls (244.93,318.96) and (243.11,318.01) .. (242.57,316.3) .. controls (242.04,314.59) and (242.99,312.78) .. (244.7,312.24) -- cycle ;

% Text Node
\draw (139,105.4) node [anchor=north west][inner sep=0.75pt]  [font=\normalsize]  {$\mathrm{RI}$};
% Text Node
\draw (254,106.4) node [anchor=north west][inner sep=0.75pt]    {$\mathrm{RII}$};
% Text Node
\draw (390,105.4) node [anchor=north west][inner sep=0.75pt]    {$\mathrm{RIII}$};
% Text Node
\draw (98.5,235.4) node [anchor=north west][inner sep=0.75pt]    {$\mathrm{VRI}$};
% Text Node
\draw (239.5,235.4) node [anchor=north west][inner sep=0.75pt]    {$\mathrm{VRII}$};
% Text Node
\draw (389,235.4) node [anchor=north west][inner sep=0.75pt]    {$\mathrm{VRIII}$};
% Text Node
\draw (268,365.4) node [anchor=north west][inner sep=0.75pt]    {$\mathrm{VRIV}$};

\end{tikzpicture}
\caption{Generalized Reidemeister moves.} \label{fig:g_reidemeister_moves}
\end{figure}

In an oriented virtual link diagram, a classical crossing is {\it positive} if the co-orientation of the over-arc and that of the under-arc match with the orientation of the plane, otherwise it is {\it negative}. Positive and negative crossings are depicted in Figure \ref{fig:positive_crossing_and_negative_crossing}.

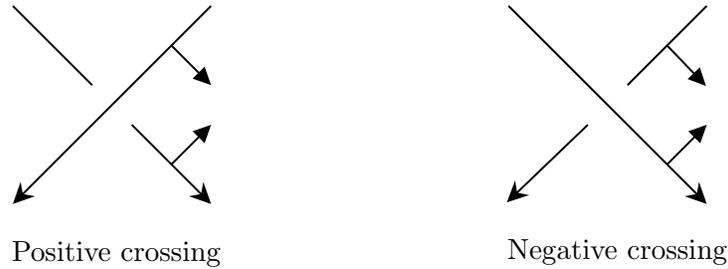
\begin{figure}[H]

\tikzset{every picture/.style={line width=0.75pt}} %set default line width to 0.75pt        

\begin{tikzpicture}[x=0.75pt,y=0.75pt,yscale=-1,xscale=1]

%Straight Lines [id:da22224052496501134] 
\draw    (50,50) -- (90,90) ;
%Straight Lines [id:da4076911332826282] 
\draw    (150,50) -- (52.12,147.88) ;
\draw [shift={(50,150)}, rotate = 315] [fill={rgb, 255:red, 0; green, 0; blue, 0 }  ][line width=0.08]  [draw opacity=0] (10.72,-5.15) -- (0,0) -- (10.72,5.15) -- (7.12,0) -- cycle    ;
%Straight Lines [id:da2648046585702869] 
\draw    (110,110) -- (147.88,147.88) ;
\draw [shift={(150,150)}, rotate = 225] [fill={rgb, 255:red, 0; green, 0; blue, 0 }  ][line width=0.08]  [draw opacity=0] (10.72,-5.15) -- (0,0) -- (10.72,5.15) -- (7.12,0) -- cycle    ;
%Straight Lines [id:da5185526157214335] 
\draw    (130,70) -- (147.88,87.88) ;
\draw [shift={(150,90)}, rotate = 225] [fill={rgb, 255:red, 0; green, 0; blue, 0 }  ][line width=0.08]  [draw opacity=0] (8.93,-4.29) -- (0,0) -- (8.93,4.29) -- cycle    ;
%Straight Lines [id:da4369250360078193] 
\draw    (130,130) -- (147.88,112.12) ;
\draw [shift={(150,110)}, rotate = 135] [fill={rgb, 255:red, 0; green, 0; blue, 0 }  ][line width=0.08]  [draw opacity=0] (8.93,-4.29) -- (0,0) -- (8.93,4.29) -- cycle    ;
%Straight Lines [id:da33052530612697284] 
\draw    (300,50) -- (397.88,147.88) ;
\draw [shift={(400,150)}, rotate = 225] [fill={rgb, 255:red, 0; green, 0; blue, 0 }  ][line width=0.08]  [draw opacity=0] (10.72,-5.15) -- (0,0) -- (10.72,5.15) -- (7.12,0) -- cycle    ;
%Straight Lines [id:da16803547774207217] 
\draw    (380,130) -- (397.88,112.12) ;
\draw [shift={(400,110)}, rotate = 135] [fill={rgb, 255:red, 0; green, 0; blue, 0 }  ][line width=0.08]  [draw opacity=0] (8.93,-4.29) -- (0,0) -- (8.93,4.29) -- cycle    ;
%Straight Lines [id:da6821689919043684] 
\draw    (340,110.01) -- (301.35,147.1) ;
\draw [shift={(299.18,149.18)}, rotate = 316.18] [fill={rgb, 255:red, 0; green, 0; blue, 0 }  ][line width=0.08]  [draw opacity=0] (10.72,-5.15) -- (0,0) -- (10.72,5.15) -- (7.12,0) -- cycle    ;
%Straight Lines [id:da12653840452226017] 
\draw    (380,70) -- (397.51,88.24) ;
\draw [shift={(399.58,90.41)}, rotate = 226.18] [fill={rgb, 255:red, 0; green, 0; blue, 0 }  ][line width=0.08]  [draw opacity=0] (8.93,-4.29) -- (0,0) -- (8.93,4.29) -- cycle    ;
%Straight Lines [id:da2876395968001335] 
\draw    (400,50) -- (360,90) ;

% Text Node
\draw (48,167) node [anchor=north west][inner sep=0.75pt]   [align=left] {Positive crossing};
% Text Node
\draw (298,166) node [anchor=north west][inner sep=0.75pt]   [align=left] {Negative crossing};

\end{tikzpicture}
\caption{Positive and negative crossings.} \label{fig:positive_crossing_and_negative_crossing}
\end{figure}

Consider an oriented virtual link $L$, represented by a virtual link diagram $D$. In $D$, a {\it semiarc} is defined as an arc between two consecutive double points. The associated {\it fundamental biquandle} $BQ(L)$ is the quotient of the free biquandle generated by the labels on the set of semiarcs in $D$. This is done by applying the equivalence relation induced by the relations at crossing in $D$, as shown in Figure \ref{fig:biquandle_crossing_rules}. This association is independent of the diagrams used to represent $L$. For more details, we refer the readers to \cite[Section 2.6]{MR2191942} and \cite{MR2384830,MR2100870}.

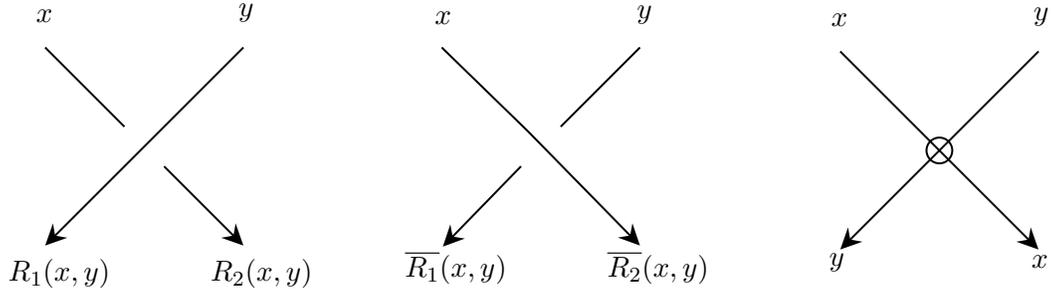
\begin{figure}[H]

\tikzset{every picture/.style={line width=0.75pt}} %set default line width to 0.75pt        

\begin{tikzpicture}[x=0.75pt,y=0.75pt,yscale=-1,xscale=1]

%Straight Lines [id:da7676984726571049] 
\draw    (50,50) -- (90,90) ;
%Straight Lines [id:da6101632830824695] 
\draw    (110,110) -- (147.88,147.88) ;
\draw [shift={(150,150)}, rotate = 225] [fill={rgb, 255:red, 0; green, 0; blue, 0 }  ][line width=0.08]  [draw opacity=0] (10.72,-5.15) -- (0,0) -- (10.72,5.15) -- (7.12,0) -- cycle    ;
%Straight Lines [id:da5973922257698383] 
\draw    (150,50) -- (52.12,147.88) ;
\draw [shift={(50,150)}, rotate = 315] [fill={rgb, 255:red, 0; green, 0; blue, 0 }  ][line width=0.08]  [draw opacity=0] (10.72,-5.15) -- (0,0) -- (10.72,5.15) -- (7.12,0) -- cycle    ;
%Straight Lines [id:da5168385998660995] 
\draw    (252.65,148.4) -- (290,110) ;
\draw [shift={(250.56,150.55)}, rotate = 314.21] [fill={rgb, 255:red, 0; green, 0; blue, 0 }  ][line width=0.08]  [draw opacity=0] (10.72,-5.15) -- (0,0) -- (10.72,5.15) -- (7.12,0) -- cycle    ;
%Straight Lines [id:da7454981991127962] 
\draw    (310,90) -- (350,50) ;
%Straight Lines [id:da27993052107042304] 
\draw    (250,50) -- (293.37,92.19) -- (347.9,147.86) ;
\draw [shift={(350,150)}, rotate = 225.59] [fill={rgb, 255:red, 0; green, 0; blue, 0 }  ][line width=0.08]  [draw opacity=0] (10.72,-5.15) -- (0,0) -- (10.72,5.15) -- (7.12,0) -- cycle    ;
%Straight Lines [id:da8816054440239901] 
\draw    (451,52) -- (548.88,149.88) ;
\draw [shift={(551,152)}, rotate = 225] [fill={rgb, 255:red, 0; green, 0; blue, 0 }  ][line width=0.08]  [draw opacity=0] (10.72,-5.15) -- (0,0) -- (10.72,5.15) -- (7.12,0) -- cycle    ;
%Straight Lines [id:da6858332421095021] 
\draw    (551,52) -- (453.12,149.88) ;
\draw [shift={(451,152)}, rotate = 315] [fill={rgb, 255:red, 0; green, 0; blue, 0 }  ][line width=0.08]  [draw opacity=0] (10.72,-5.15) -- (0,0) -- (10.72,5.15) -- (7.12,0) -- cycle    ;
%Shape: Circle [id:dp47118650332955025] 
\draw   (494.5,102) .. controls (494.5,98.41) and (497.41,95.5) .. (501,95.5) .. controls (504.59,95.5) and (507.5,98.41) .. (507.5,102) .. controls (507.5,105.59) and (504.59,108.5) .. (501,108.5) .. controls (497.41,108.5) and (494.5,105.59) .. (494.5,102) -- cycle ;

% Text Node
\draw (44,29.4) node [anchor=north west][inner sep=0.75pt]    {$x$};
% Text Node
\draw (146,26.4) node [anchor=north west][inner sep=0.75pt]    {$y$};
% Text Node
\draw (30,155.4) node [anchor=north west][inner sep=0.75pt]    {$R_{1}( x,y)$};
% Text Node
\draw (132,155.4) node [anchor=north west][inner sep=0.75pt]    {$R_{2}( x,y)$};
% Text Node
\draw (245,30.4) node [anchor=north west][inner sep=0.75pt]    {$x$};
% Text Node
\draw (347,27.4) node [anchor=north west][inner sep=0.75pt]    {$y$};
% Text Node
\draw (230,151.4) node [anchor=north west][inner sep=0.75pt]    {$\overline{R_{1}}( x,y)$};
% Text Node
\draw (332,151.4) node [anchor=north west][inner sep=0.75pt]    {$\overline{R_{2}}( x,y)$};
% Text Node
\draw (445,31.4) node [anchor=north west][inner sep=0.75pt]    {$x$};
% Text Node
\draw (547,28.4) node [anchor=north west][inner sep=0.75pt]    {$y$};
% Text Node
\draw (444,152.4) node [anchor=north west][inner sep=0.75pt]    {$y$};
% Text Node
\draw (546,153.4) node [anchor=north west][inner sep=0.75pt]    {$x$};

\end{tikzpicture}

\caption{Labeling rules for biquandle.}
\label{fig:biquandle_crossing_rules}
\end{figure}

In a similar way, one can associate a {\it fundamental virtual biquandle} $(VBQ(L), f , R)$ to $L$. This virtual biquandle is defined as the quotient of the free virtual biquandle generated by the labels on the set of semiarcs in the virtual link diagram $D$. This  is done by applying the equivalence relation induced by the relations at the crossings in $D$, as shown in Figure \ref{fig:old_virtual_crossing_rules}. Importantly, the virtual biquandle $(VBQ(L), f, R)$ is independent of the diagram $D$. See \cite[Section 2.6]{MR2191942} for more details.

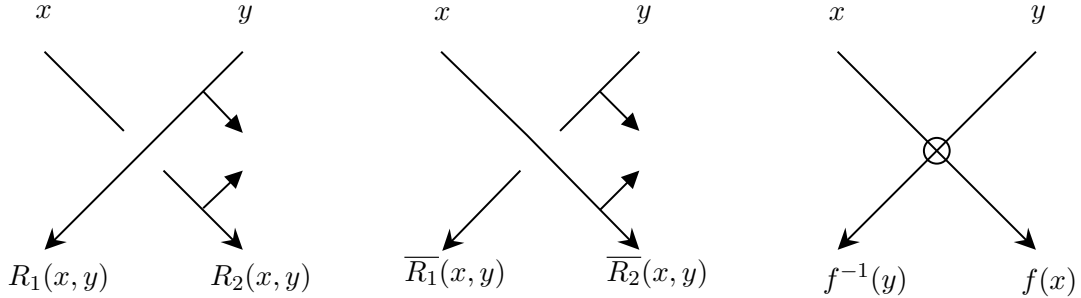
\begin{figure}[H]

\tikzset{every picture/.style={line width=0.75pt}} %set default line width to 0.75pt        

\tikzset{every picture/.style={line width=0.75pt}} %set default line width to 0.75pt        

\begin{tikzpicture}[x=0.75pt,y=0.75pt,yscale=-1,xscale=1]

%Straight Lines [id:da7676984726571049] 
\draw    (50,50) -- (90,90) ;
%Straight Lines [id:da6101632830824695] 
\draw    (110,110) -- (147.88,147.88) ;
\draw [shift={(150,150)}, rotate = 225] [fill={rgb, 255:red, 0; green, 0; blue, 0 }  ][line width=0.08]  [draw opacity=0] (10.72,-5.15) -- (0,0) -- (10.72,5.15) -- (7.12,0) -- cycle    ;
%Straight Lines [id:da5973922257698383] 
\draw    (150,50) -- (52.12,147.88) ;
\draw [shift={(50,150)}, rotate = 315] [fill={rgb, 255:red, 0; green, 0; blue, 0 }  ][line width=0.08]  [draw opacity=0] (10.72,-5.15) -- (0,0) -- (10.72,5.15) -- (7.12,0) -- cycle    ;
%Straight Lines [id:da5168385998660995] 
\draw    (252.65,148.4) -- (290,110) ;
\draw [shift={(250.56,150.55)}, rotate = 314.21] [fill={rgb, 255:red, 0; green, 0; blue, 0 }  ][line width=0.08]  [draw opacity=0] (10.72,-5.15) -- (0,0) -- (10.72,5.15) -- (7.12,0) -- cycle    ;
%Straight Lines [id:da7454981991127962] 
\draw    (310,90) -- (350,50) ;
%Straight Lines [id:da27993052107042304] 
\draw    (250,50) -- (293.37,92.19) -- (347.9,147.86) ;
\draw [shift={(350,150)}, rotate = 225.59] [fill={rgb, 255:red, 0; green, 0; blue, 0 }  ][line width=0.08]  [draw opacity=0] (10.72,-5.15) -- (0,0) -- (10.72,5.15) -- (7.12,0) -- cycle    ;
%Straight Lines [id:da9057306701040236] 
\draw    (450,50) -- (547.88,147.88) ;
\draw [shift={(550,150)}, rotate = 225] [fill={rgb, 255:red, 0; green, 0; blue, 0 }  ][line width=0.08]  [draw opacity=0] (10.72,-5.15) -- (0,0) -- (10.72,5.15) -- (7.12,0) -- cycle    ;
%Straight Lines [id:da3440324253062397] 
\draw    (550,50) -- (452.12,147.88) ;
\draw [shift={(450,150)}, rotate = 315] [fill={rgb, 255:red, 0; green, 0; blue, 0 }  ][line width=0.08]  [draw opacity=0] (10.72,-5.15) -- (0,0) -- (10.72,5.15) -- (7.12,0) -- cycle    ;
%Shape: Circle [id:dp8057867317746894] 
\draw   (493.5,100) .. controls (493.5,96.41) and (496.41,93.5) .. (500,93.5) .. controls (503.59,93.5) and (506.5,96.41) .. (506.5,100) .. controls (506.5,103.59) and (503.59,106.5) .. (500,106.5) .. controls (496.41,106.5) and (493.5,103.59) .. (493.5,100) -- cycle ;
%Straight Lines [id:da07245464309133443] 
\draw    (130,70) -- (147.93,88.83) ;
\draw [shift={(150,91)}, rotate = 226.4] [fill={rgb, 255:red, 0; green, 0; blue, 0 }  ][line width=0.08]  [draw opacity=0] (8.93,-4.29) -- (0,0) -- (8.93,4.29) -- cycle    ;
%Straight Lines [id:da8699460613647552] 
\draw    (330,130) -- (347.88,112.12) ;
\draw [shift={(350,110)}, rotate = 135] [fill={rgb, 255:red, 0; green, 0; blue, 0 }  ][line width=0.08]  [draw opacity=0] (8.93,-4.29) -- (0,0) -- (8.93,4.29) -- cycle    ;
%Straight Lines [id:da9525155138284294] 
\draw    (130,129) -- (147.83,112.07) ;
\draw [shift={(150,110)}, rotate = 136.47] [fill={rgb, 255:red, 0; green, 0; blue, 0 }  ][line width=0.08]  [draw opacity=0] (8.93,-4.29) -- (0,0) -- (8.93,4.29) -- cycle    ;
%Straight Lines [id:da013934537338040287] 
\draw    (330,70) -- (347.88,87.88) ;
\draw [shift={(350,90)}, rotate = 225] [fill={rgb, 255:red, 0; green, 0; blue, 0 }  ][line width=0.08]  [draw opacity=0] (8.93,-4.29) -- (0,0) -- (8.93,4.29) -- cycle    ;

% Text Node
\draw (44,25.4) node [anchor=north west][inner sep=0.75pt]    {$x$};
% Text Node
\draw (146,25.4) node [anchor=north west][inner sep=0.75pt]    {$y$};
% Text Node
\draw (30,155.4) node [anchor=north west][inner sep=0.75pt]    {$R_{1}( x,y)$};
% Text Node
\draw (132,155.4) node [anchor=north west][inner sep=0.75pt]    {$R_{2}( x,y)$};
% Text Node
\draw (245,26.4) node [anchor=north west][inner sep=0.75pt]    {$x$};
% Text Node
\draw (347,25.4) node [anchor=north west][inner sep=0.75pt]    {$y$};
% Text Node
\draw (230,151.4) node [anchor=north west][inner sep=0.75pt]    {$\overline{R_{1}}( x,y)$};
% Text Node
\draw (332,151.4) node [anchor=north west][inner sep=0.75pt]    {$\overline{R_{2}}( x,y)$};
% Text Node
\draw (444,26.4) node [anchor=north west][inner sep=0.75pt]    {$x$};
% Text Node
\draw (546,25.4) node [anchor=north west][inner sep=0.75pt]    {$y$};
% Text Node
\draw (441,155.4) node [anchor=north west][inner sep=0.75pt]    {$f^{-1}( y)$};
% Text Node
\draw (541,157.4) node [anchor=north west][inner sep=0.75pt]    {$f( x)$};

\end{tikzpicture}

\caption{Labeling rules for virtual biquandle.} \label{fig:old_virtual_crossing_rules}
\end{figure}

A {\it coloring} of a virtual link $L$ by a virtual biquandle $(X,R',f')$ is a virtual biquandle homomorphism $\phi:(VBQ(L), f, R) \to (X,R',f')$. Note that labeling the semiarcs of the diagram $D$ with elements from $(X,R',f')$, following the rules depicted in Figure \ref{fig:old_virtual_crossing_rules}, results in a coloring of $L$.

%%%%%%%%%%%%%%%%%%%%%%%%%%%%%%%%%%%%%%%%%%%%%%%%%%%%%%%%%%%%%%%%%%%%%%%%%%%%%%%%%%%%%%%%%%%%%%%%%%%%%%%%%%%%%%%%%%%%%%%%%%%%%%%%%%%%%%%%%%%%%%%%%%%%%%%%%%%%

%%%%%%%%%%%%%%%%%%%%%%%%%%%%%%%%%%%%%%%%%%%%%%%%%%%%%%%%%%%%%%%
%%%%%%%%%%%%%%%%%%%%%%%%%%%%%%%%%%%%%%%%%%%%%%%%%%%%%%%%%%%%%%%%%%%%%%%%%%%%%%%%%%%%%%%%

\section{Set theoretic bijection between the set of colorings by biquandles and virtual biquandles}\label{Col}
In this section, we prove that if $L$ is a virtual link, then the set of colorings of $L$ by a virtual biquandle $(X,f,R)$ under the rules illustrated in Figure \ref{fig:old_virtual_crossing_rules} is in bijection with the set of colorings of $L$ by $(X,VR)$ under the rules illustrated in Figure \ref{fig:biquandle_crossing_rules}.

First recall the definition of the virtual braid group $VB_n$ on $n$-strands, which has the following presentation (see \cite[Theorem 4]{MR2493369}):

\begin{itemize}
\item Generators: $\sigma_1, \sigma_2, \ldots, \sigma_{n-1},\rho_1, \rho_2, \ldots, \rho_{n-1}.$
\item Relations: 
\begin{align*}
&\sigma_i \sigma_{i+1} \sigma_i = \sigma_{i+1} \sigma_i \sigma_{i+1}, ~i=1,2,\ldots, n-2;\\
&\sigma_i \sigma_j = \sigma_j \sigma_i, ~|i-j| \geq 2;\\
&\rho_i\rho_{i+1} \rho_i = \rho_{i+1}\rho_i \rho_{i+1}, ~ i=1,2,\ldots,n-2;\\
&\rho_i \rho_j =\rho_j \rho_i,~ | i-j| \geq 2;\\
&\rho_i^2 =1,~ i=1,2,\ldots,n-1;\\
&\sigma_i \rho_j = \rho_j \sigma_i,~ |i-j| \geq 2;\\
&\rho_i \rho_{i+1} \sigma_i = \sigma_{i+1} \rho_i \rho_{i+1},~ i=1,2,\ldots, n-2.
\end{align*}
\end{itemize}
Geometrically, $\sigma_i$, $\sigma_i^{-1}$ and $\rho_i$ are shown in Figure \ref{fig:generators_of_VB_n}.

\begin{figure}[H]

\tikzset{every picture/.style={line width=0.75pt}} %set default line width to 0.75pt        

\begin{tikzpicture}[x=0.75pt,y=0.75pt,yscale=-1,xscale=1]

%Straight Lines [id:da4603696225511227] 
\draw    (50,51) -- (70,71) ;
%Straight Lines [id:da6842256004005207] 
\draw    (330,70) -- (350,50) ;
%Straight Lines [id:da7449735670090046] 
\draw    (100,51) -- (52.12,98.88) ;
\draw [shift={(50,101)}, rotate = 315] [fill={rgb, 255:red, 0; green, 0; blue, 0 }  ][line width=0.08]  [draw opacity=0] (10.72,-5.15) -- (0,0) -- (10.72,5.15) -- (7.12,0) -- cycle    ;
%Straight Lines [id:da6107505852726666] 
\draw    (300,50) -- (347.88,97.88) ;
\draw [shift={(350,100)}, rotate = 225] [fill={rgb, 255:red, 0; green, 0; blue, 0 }  ][line width=0.08]  [draw opacity=0] (10.72,-5.15) -- (0,0) -- (10.72,5.15) -- (7.12,0) -- cycle    ;
%Straight Lines [id:da11355085692491307] 
\draw    (320,80) -- (302.12,97.88) ;
\draw [shift={(300,100)}, rotate = 315] [fill={rgb, 255:red, 0; green, 0; blue, 0 }  ][line width=0.08]  [draw opacity=0] (10.72,-5.15) -- (0,0) -- (10.72,5.15) -- (7.12,0) -- cycle    ;
%Straight Lines [id:da725830716813222] 
\draw    (80,81) -- (97.88,98.88) ;
\draw [shift={(100,101)}, rotate = 225] [fill={rgb, 255:red, 0; green, 0; blue, 0 }  ][line width=0.08]  [draw opacity=0] (10.72,-5.15) -- (0,0) -- (10.72,5.15) -- (7.12,0) -- cycle    ;
%Straight Lines [id:da21054449233049755] 
\draw    (550,50) -- (597.88,97.88) ;
\draw [shift={(600,100)}, rotate = 225] [fill={rgb, 255:red, 0; green, 0; blue, 0 }  ][line width=0.08]  [draw opacity=0] (10.72,-5.15) -- (0,0) -- (10.72,5.15) -- (7.12,0) -- cycle    ;
%Straight Lines [id:da9192939725719612] 
\draw    (600,50) -- (552.12,97.88) ;
\draw [shift={(550,100)}, rotate = 315] [fill={rgb, 255:red, 0; green, 0; blue, 0 }  ][line width=0.08]  [draw opacity=0] (10.72,-5.15) -- (0,0) -- (10.72,5.15) -- (7.12,0) -- cycle    ;
%Straight Lines [id:da6112048739492322] 
\draw    (500,50) -- (500,97) ;
\draw [shift={(500,100)}, rotate = 270] [fill={rgb, 255:red, 0; green, 0; blue, 0 }  ][line width=0.08]  [draw opacity=0] (10.72,-5.15) -- (0,0) -- (10.72,5.15) -- (7.12,0) -- cycle    ;
%Straight Lines [id:da8116667361623648] 
\draw    (650,50) -- (650,97) ;
\draw [shift={(650,100)}, rotate = 270] [fill={rgb, 255:red, 0; green, 0; blue, 0 }  ][line width=0.08]  [draw opacity=0] (10.72,-5.15) -- (0,0) -- (10.72,5.15) -- (7.12,0) -- cycle    ;
%Straight Lines [id:da698254877884811] 
\draw    (250,50) -- (250,97) ;
\draw [shift={(250,100)}, rotate = 270] [fill={rgb, 255:red, 0; green, 0; blue, 0 }  ][line width=0.08]  [draw opacity=0] (10.72,-5.15) -- (0,0) -- (10.72,5.15) -- (7.12,0) -- cycle    ;
%Straight Lines [id:da3613096419122037] 
\draw    (400,50) -- (400,97) ;
\draw [shift={(400,100)}, rotate = 270] [fill={rgb, 255:red, 0; green, 0; blue, 0 }  ][line width=0.08]  [draw opacity=0] (10.72,-5.15) -- (0,0) -- (10.72,5.15) -- (7.12,0) -- cycle    ;
%Straight Lines [id:da9706821806752535] 
\draw    (150,50) -- (150,97) ;
\draw [shift={(150,100)}, rotate = 270] [fill={rgb, 255:red, 0; green, 0; blue, 0 }  ][line width=0.08]  [draw opacity=0] (10.72,-5.15) -- (0,0) -- (10.72,5.15) -- (7.12,0) -- cycle    ;
%Straight Lines [id:da7089211908428332] 
\draw    (0,50) -- (0,97) ;
\draw [shift={(0,100)}, rotate = 270] [fill={rgb, 255:red, 0; green, 0; blue, 0 }  ][line width=0.08]  [draw opacity=0] (10.72,-5.15) -- (0,0) -- (10.72,5.15) -- (7.12,0) -- cycle    ;
%Shape: Circle [id:dp7489633014115193] 
\draw   (570.36,75) .. controls (570.36,72.44) and (572.44,70.36) .. (575,70.36) .. controls (577.56,70.36) and (579.64,72.44) .. (579.64,75) .. controls (579.64,77.56) and (577.56,79.64) .. (575,79.64) .. controls (572.44,79.64) and (570.36,77.56) .. (570.36,75) -- cycle ;

% Text Node
\draw (497,32.4) node [anchor=north west][inner sep=0.75pt]  [font=\scriptsize]  {$1$};
% Text Node
\draw (548,32.4) node [anchor=north west][inner sep=0.75pt]  [font=\scriptsize]  {$i$};
% Text Node
\draw (588,33.4) node [anchor=north west][inner sep=0.75pt]  [font=\scriptsize]  {$i+1$};
% Text Node
\draw (513,62.4) node [anchor=north west][inner sep=0.75pt]  [font=\normalsize]  {$\dotsc $};
% Text Node
\draw (614,62.4) node [anchor=north west][inner sep=0.75pt]  [font=\normalsize]  {$\dotsc $};
% Text Node
\draw (647,32.4) node [anchor=north west][inner sep=0.75pt]  [font=\scriptsize]  {$n$};
% Text Node
\draw (247,32.4) node [anchor=north west][inner sep=0.75pt]  [font=\scriptsize]  {$1$};
% Text Node
\draw (263,62.4) node [anchor=north west][inner sep=0.75pt]  [font=\normalsize]  {$\dotsc $};
% Text Node
\draw (364,62.4) node [anchor=north west][inner sep=0.75pt]  [font=\normalsize]  {$\dotsc $};
% Text Node
\draw (397,32.4) node [anchor=north west][inner sep=0.75pt]  [font=\scriptsize]  {$n$};
% Text Node
\draw (299,32.4) node [anchor=north west][inner sep=0.75pt]  [font=\scriptsize]  {$i$};
% Text Node
\draw (339,33.4) node [anchor=north west][inner sep=0.75pt]  [font=\scriptsize]  {$i+1$};
% Text Node
\draw (114,62.4) node [anchor=north west][inner sep=0.75pt]  [font=\normalsize]  {$\dotsc $};
% Text Node
\draw (147,32.4) node [anchor=north west][inner sep=0.75pt]  [font=\scriptsize]  {$n$};
% Text Node
\draw (-3,32.4) node [anchor=north west][inner sep=0.75pt]  [font=\scriptsize]  {$1$};
% Text Node
\draw (13,62.4) node [anchor=north west][inner sep=0.75pt]  [font=\normalsize]  {$\dotsc $};
% Text Node
\draw (49,33.4) node [anchor=north west][inner sep=0.75pt]  [font=\scriptsize]  {$i$};
% Text Node
\draw (89,34.4) node [anchor=north west][inner sep=0.75pt]  [font=\scriptsize]  {$i+1$};
% Text Node
\draw (68,103.4) node [anchor=north west][inner sep=0.75pt]  [font=\normalsize]  {$\sigma _{i}$};
% Text Node
\draw (318,101.4) node [anchor=north west][inner sep=0.75pt]    {$\sigma _{i}^{-1}$};
% Text Node
\draw (571,102.4) node [anchor=north west][inner sep=0.75pt]    {$\rho _{i}$};

\end{tikzpicture}
\caption{Generators of virtual braid group $VB_n$.} \label{fig:generators_of_VB_n}
\end{figure}
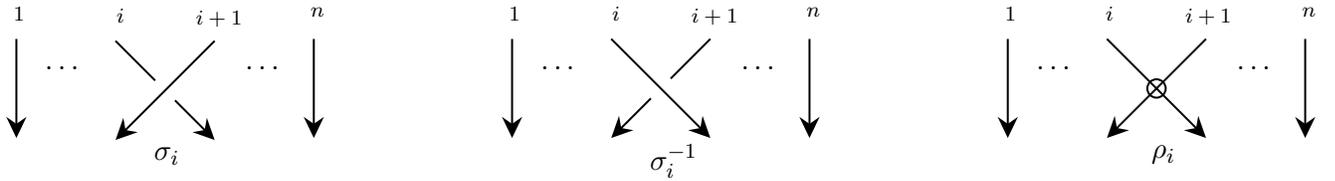

For $\beta_1, \beta_2 \in VB_n$, the composition $\beta_2 \beta_1$ in $VB_n$ represents the braid shown in Figure \ref{fig:composition_of_braids}.

\begin{figure}[H]
\begin{center}
\includegraphics[height=1in,width=2in,angle=00]{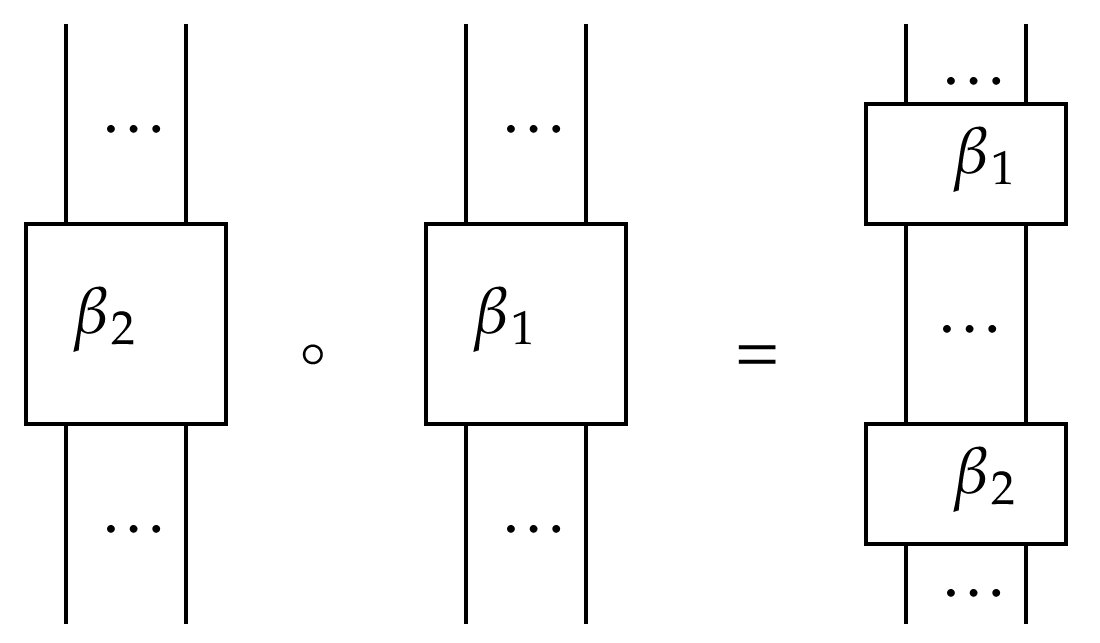}

\end{center}
\caption{Composition operation on $VB_n$.}
\label{fig:composition_of_braids}
\end{figure}

Let $VB_n$ be the virtual braid group on $n$-strands and $(X, f, R)$ be a virtual biquandle. Then $VB_n$ acts on the product $X^n$ of $n$ copies of $X$ by the representation
\[
\phi: VB_n \to \Aut(X^n)
\]
defined as follows:
\begin{align*}
	&\phi(\sigma_i) :
	\left\{
	\begin{array}{l}
		x_i \mapsto  R_1(x_i,x_{i+1}), \\
		x_{i+1} \mapsto R_2(x_i,x_{i+1}),  \\
		x_j \mapsto x_j, \textrm{ for } j \neq i, i+1, \\
	\end{array}
	\right.~~~
	&\phi(\rho_i) :
	\left\{
	\begin{array}{l}
		x_i \mapsto  f^{-1}(x_{i+1}), \\
		x_{i+1} \mapsto ~~f(x_i), \\
	\end{array}
	\right.
\end{align*}
where, $1 \leq i \leq n-1$. It is easy to check that
\begin{align*}
&\phi(\sigma_i^{-1}) :
	\left\{	\begin{array}{l}
		x_i \mapsto  \overbar{R_1}(x_i,x_{i+1}), \\
		x_{i+1} \mapsto ~~\overbar{R_2}(x_i,x_{i+1}). \\
	\end{array}
\right.
\end{align*}
Now, consider an automorphism \[\theta: X^n \to X^n\] defined as 
\[\theta = f^{n-1} \times f^{n-2} \times \cdots \times f^{n-i} \times \cdots \times f \times \id.\] Note that 
\[\theta^{-1} = f^{-n+1} \times f^{-n+2} \times \cdots \times f^{-n+i} \times \cdots \times f^{-1} \times \id.\]
Now define a new action $$\psi: VB_n \to \Aut(X^n)$$ of $VB_n$ on $X^n$ given by $$\psi(\beta)= \theta \circ \phi (\beta) \circ \theta^{-1}.$$ We give the following explicit images of the generators of $VB_n$ under the representation $\psi$:

\begin{align*}
	&\psi(\sigma_i) :
	\left\{
	\begin{array}{l}
		x_i \mapsto  R_1\big(x_i,f(x_{i+1})\big), \\
		x_{i+1} \mapsto R_2\big(f^{-1}(x_i),x_{i+1}\big),  \\
		x_j \mapsto x_j, \textrm{ for } j \neq i, i+1, \\
	\end{array}
	\right.~~~
	&\psi(\rho_i) :
	\left\{
	\begin{array}{l}
		x_i \mapsto  x_{i+1}, \\
		x_{i+1} \mapsto ~~x_i \\
	\end{array}
	\right.
\end{align*}
and

\begin{align*}
&\psi(\sigma_i^{-1}) :
	\left\{	\begin{array}{l}
		x_i \mapsto  \overbar{R_1}\big(x_i,f(x_{i+1})\big), \\
		x_{i+1} \mapsto ~~\overbar{R_2}\big(f^{-1}(x_i),x_{i+1}\big). \\
	\end{array}
\right.
\end{align*}

\begin{lemma}\label{lem:number_of_colorings_are_same}
Let $\beta \in VB_n$ and $(X,f,R)$ be a virtual biquandle. Then there is a bijection between the fixed point set of $\phi(\beta)$ and $\psi(\beta)$.
\end{lemma}
\begin{proof}
It is sufficient to observe that if $\phi(\beta)\big( (x_1, \ldots, x_n) \big) =(x_1, \ldots, x_n)$, then $\psi(\beta) \big(\theta( (x_1, \ldots, x_n) ) \big) = \theta( (x_1, \ldots, x_n) )$, and $\theta$ is a bijection.
\end{proof}

By \cite{MR2351010}, for every virtual link $L$ there exists a virtual braid $\beta$ whose closure is $L$.

\begin{theorem}\label{thm:number_of_colorings_are_same}
Suppose $(X,f,R)$ is a virtual biquandle and $L$ a virtual link diagram. Then the set of colorings of $L$ by $(X,f,R)$, under the virtual biquandle coloring rule $($see Figure \ref{fig:old_virtual_crossing_rules}$)$, is in bijection with the set of colorings of $L$ by $(X,VR)$, under the biquandle coloring rule $($see Figure \ref{fig:biquandle_crossing_rules}$)$.
\end{theorem}
\begin{proof}
Let $\beta \in VB_n$ such that its closure is $L$. Then the set of colorings of $L$ by $(X,f,R)$ under the virtual biquandle rule $($see Figure \ref{fig:old_virtual_crossing_rules}$)$ and by $(X, VR)$ under the biquandle coloring rule $($see Figure \ref{fig:biquandle_crossing_rules}$)$ are the fixed points of $\phi(\beta)$ and $\psi(\beta)$, respectively. By Lemma \ref{lem:number_of_colorings_are_same}, and noting that the set of colorings of virtual link diagrams related by the generalized Reidemeister moves (Figure \ref{fig:g_reidemeister_moves}) are always in bijection, the result follows.
\end{proof}

As a result of Theorem \ref{thm:number_of_colorings_are_same} we have the following corollary.

\begin{corollary}\label{cor:cor_number_of_colorings_are_same}
Let $L$ be a virtual knot and $(X,f,R)$ a finite virtual biquandle. Then the number of colorings of $L$ by $(X,f,R)$ under the rules shown in Figure \ref{fig:old_virtual_crossing_rules} is equal to the number of colorings of $L$ by $(X, VR)$ under the rules shown in Figure \ref{fig:biquandle_crossing_rules}.
\end{corollary}

Thus, the number of colorings of a virtual link $L$ by a virtual biquandle can be retrieved from  colorings of $L$ by a suitable biquandle.
\section{Deriving Fundamental Virtual Biquandle from Classical Crossings}\label{Relation}

Let $(FVBQ_n,f,R)$ denotes the free virtual biquandle on the set $\{x_1, x_2, \ldots, x_n\}$. The virtual braid group $VB_n$ on $n$-strands acts on $FVBQ_n$ via the representation $$\Phi: VB_n \to \Aut\big((FVBQ_n, f, R)\big)$$ defined as follows:

\begin{align*}
	&\Phi(\sigma_i) :
	\left\{
	\begin{array}{l}
		x_i \mapsto  R_1(x_i,x_{i+1}), \\
		x_{i+1} \mapsto R_2(x_i,x_{i+1}),  \\
		x_j \mapsto x_j, \textrm{ for } j \neq i, i+1, \\
	\end{array}
	\right.~~~
	&\Phi(\rho_i) :
	\left\{
	\begin{array}{l}
		x_i \mapsto  f^{-1}(x_{i+1}), \\
		x_{i+1} \mapsto ~~f(x_i). \\
	\end{array}
	\right.
\end{align*}
Note that  for $1 \leq i \leq n-1$,
\begin{align*}
\Phi(\sigma_{i}) \circ f &= f \circ \Phi(\sigma_{i})\\
\Phi(\rho_{i}) \circ f&=f \circ \Phi(\rho_{i}).\\
\end{align*}
One can check that
\begin{align*}
&\Phi(\sigma_i^{-1}) :
	\left\{	\begin{array}{l}
		x_i \mapsto  \overbar{R_1}(x_i,x_{i+1}), \\
		x_{i+1} \mapsto ~~\overbar{R_2}(x_i,x_{i+1}). \\
	\end{array}
\right.
\end{align*}
Now, consider an automorphism $$\Theta: (FVBQ_n, f, R) \to (FVBQ_n, f, R)$$ defined as $$\Theta(x_i)=f^{n-i}(x_i).$$
Now define a new action of $VB_n$ on $(FVBQ_n, f,R)$ $$\Psi: VB_n \to \Aut((FVBQ_n, f,R))$$ given by $$\Psi(\beta)= \Theta \circ \Phi (\beta) \circ \Theta^{-1}.$$ We write down the images of the generators of $VB_n$ under the map $\Psi$:

\begin{align*}
	&\Psi(\sigma_i) :
	\left\{
	\begin{array}{l}
		x_i \mapsto  R_1\big(x_i,f(x_{i+1})\big), \\
		x_{i+1} \mapsto R_2\big(f^{-1}(x_i),x_{i+1}\big),  \\
		x_j \mapsto x_j, \textrm{ for } j \neq i, i+1, \\
	\end{array}
	\right.~~~
	&\Psi(\rho_i) :
	\left\{
	\begin{array}{l}
		x_i \mapsto  x_{i+1}, \\
		x_{i+1} \mapsto ~~x_i \\
	\end{array}
	\right.
\end{align*}
and

\begin{align*}
&\Psi(\sigma_i^{-1}) :
	\left\{	\begin{array}{l}
		x_i \mapsto  \overbar{R_1}\big(x_i,f(x_{i+1})\big), \\
		x_{i+1} \mapsto ~~\overbar{R_2}\big(f^{-1}(x_i),x_{i+1}\big). \\
	\end{array}
\right.
\end{align*}
For a given $\beta \in VB_n$, associate a virtual biquandle $VBQ(\Phi,\beta) = FVBQ_n/\mathcal{R}$, where $\mathcal{R}$ is the equivalence relation generated by the set $\{ \Phi(\beta)(x_i)=x_i~|~ 1 \leq i \leq n\}$. We write $VBQ(\Phi,\beta)$ in a presentation form as follows:

$$
\langle x_1, \ldots, x_n, f, R~|~\Phi(\beta)(x_i)=x_i\rangle.
$$

Let $\beta \in VB_n$, and $\cl(\beta)$ denotes the closure of $\beta$.

\begin{theorem}
Let $\beta,\beta' \in VB_n$ such that $\cl(\beta)$ and $\cl(\beta')$ represents the same virtual link. Then $VBQ(\Phi,\beta)\cong VBQ(\Phi,\beta')$.
\end{theorem}
\begin{proof}
The proof follows from the fact that $VBQ(\Phi, \beta)$ and $VBQ(\Phi, \beta')$ are the fundamental virtual biquandles associated to $\cl(\beta)$ and $\cl(\beta')$ (see \cite[Section 2.6]{MR2191942} and \cite[Section 5]{MR2384830}).
\end{proof}
Thus, the associated virtual biquandle $VBQ(\Phi,\beta)$ is an invariant of virtual links.

\begin{theorem}\label{thm:both_virtual_biquandles_are_same}
Let $\beta \in VB_n$. Then the virtual biquandle $VBQ(\Psi,\beta)$ is isomorphic to the virtual biquandle $VBQ(\Phi,\beta)$. In particular, $VBQ(\Psi,\beta)$ is a virtual link invariant.
\end{theorem}
\begin{proof}
For $\beta \in VB_n$, the virtual biquandle $VBQ(\Phi,\beta)$ has a following presentation

$$
\langle x_1, \ldots, x_n, f,R~|~\Phi(\beta)(x_i)=x_i \textrm{ for } 1 \leq i \leq n \rangle.
$$

Consider the automorphism $\Theta: FVBQ_n \to FVBQ_n$ as defined above. Since $\Theta$ is an automorphism, thus

\begin{align*}
VBQ(\Phi,\beta)\cong &\langle x_1, \ldots, x_n,f, R~|~\Theta(\Phi(\beta)(x_i))=\Theta(x_i) \textrm{ for } 1 \leq i \leq n \rangle\\
\cong & \langle x_1, \ldots, x_n,f, R~|~\Theta\circ\Phi(\beta)\circ\Theta^{-1}(\Theta(x_i))=\Theta(x_i) \textrm{ for } 1 \leq i \leq n \rangle\\
\cong & \langle x_1, \ldots, x_n,f, R~|~\Psi(\beta)(\Theta(x_i))=\Theta(x_i) \textrm{ for } 1 \leq i \leq n \rangle\\
\cong & \langle x_1, \ldots, x_n,f, R~|~\Psi(\beta)(f^{n-i}(x_i))=f^{n-i}(x_i) \textrm{ for } 1 \leq i \leq n \rangle\\
\cong & \langle x_1, \ldots, x_n,f, R~|~\Psi(\beta)(x_i)=x_i \textrm{ for } 1 \leq i \leq n \rangle\\
\cong &VBQ(\Psi, \beta).
\end{align*}
This concludes the proof.
\end{proof}

Let $L$ be an oriented virtual link represented by a virtual link diagram $D$. Then we associate a virtual biquandle $(X, f', R')$ to $L$ defined as the quotient of a free virtual biquandle, which is generated by the labels on the set of semiarcs in the diagram $D$, by the equivalence relation induced by the relations at the crossings in $D$ as shown in Figure \ref{fig:new_labeling_virtual_crossing_rules}.

\begin{figure}[H]

\tikzset{every picture/.style={line width=0.75pt}} %set default line width to 0.75pt        

\begin{tikzpicture}[x=0.75pt,y=0.75pt,yscale=-1,xscale=1]

%Straight Lines [id:da9013782658873087] 
\draw    (50,50) -- (90,89.67) ;
%Straight Lines [id:da9571085300090044] 
\draw    (110,109.67) -- (147.89,147.87) ;
\draw [shift={(150,150)}, rotate = 225.24] [fill={rgb, 255:red, 0; green, 0; blue, 0 }  ][line width=0.08]  [draw opacity=0] (10.72,-5.15) -- (0,0) -- (10.72,5.15) -- (7.12,0) -- cycle    ;
%Straight Lines [id:da9919886201819011] 
\draw    (150,50) -- (52.12,147.88) ;
\draw [shift={(50,150)}, rotate = 315] [fill={rgb, 255:red, 0; green, 0; blue, 0 }  ][line width=0.08]  [draw opacity=0] (10.72,-5.15) -- (0,0) -- (10.72,5.15) -- (7.12,0) -- cycle    ;
%Straight Lines [id:da7993823664633507] 
\draw    (252.65,148.06) -- (290,109.67) ;
\draw [shift={(250.56,150.21)}, rotate = 314.21] [fill={rgb, 255:red, 0; green, 0; blue, 0 }  ][line width=0.08]  [draw opacity=0] (10.72,-5.15) -- (0,0) -- (10.72,5.15) -- (7.12,0) -- cycle    ;
%Straight Lines [id:da48922528210979777] 
\draw    (310,89.67) -- (350,50) ;
%Straight Lines [id:da3870266596574179] 
\draw    (250,50) -- (293.37,91.86) -- (347.91,147.85) ;
\draw [shift={(350,150)}, rotate = 225.76] [fill={rgb, 255:red, 0; green, 0; blue, 0 }  ][line width=0.08]  [draw opacity=0] (10.72,-5.15) -- (0,0) -- (10.72,5.15) -- (7.12,0) -- cycle    ;
%Straight Lines [id:da04412252423965257] 
\draw    (450,49.67) -- (547.88,147.88) ;
\draw [shift={(550,150)}, rotate = 225.1] [fill={rgb, 255:red, 0; green, 0; blue, 0 }  ][line width=0.08]  [draw opacity=0] (10.72,-5.15) -- (0,0) -- (10.72,5.15) -- (7.12,0) -- cycle    ;
%Straight Lines [id:da7258487649252963] 
\draw    (550,50) -- (452.12,147.88) ;
\draw [shift={(450,150)}, rotate = 315] [fill={rgb, 255:red, 0; green, 0; blue, 0 }  ][line width=0.08]  [draw opacity=0] (10.72,-5.15) -- (0,0) -- (10.72,5.15) -- (7.12,0) -- cycle    ;
%Shape: Circle [id:dp24396754347149296] 
\draw   (493.5,99.67) .. controls (493.5,96.08) and (496.41,93.17) .. (500,93.17) .. controls (503.59,93.17) and (506.5,96.08) .. (506.5,99.67) .. controls (506.5,103.26) and (503.59,106.17) .. (500,106.17) .. controls (496.41,106.17) and (493.5,103.26) .. (493.5,99.67) -- cycle ;

% Text Node
\draw (45,25.07) node [anchor=north west][inner sep=0.75pt]    {$x$};
% Text Node
\draw (145,25.73) node [anchor=north west][inner sep=0.75pt]    {$y$};
% Text Node
\draw (11.33,155.4) node [anchor=north west][inner sep=0.75pt]    {$R_{1}^{'}( x,\ f'( y))$};
% Text Node
\draw (111,152.73) node [anchor=north west][inner sep=0.75pt]    {$R_{2}^{'}\left( f'^{-1}( x) ,y\right)$};
% Text Node
\draw (320,151.73) node [anchor=north west][inner sep=0.75pt]    {$\overline{R_{2}^{'}}\left( f'^{-1}( x) ,y\right)$};
% Text Node
\draw (223,151.73) node [anchor=north west][inner sep=0.75pt]    {$\overline{R_{1}^{'}}( x,\ f'( y))$};
% Text Node
\draw (244,25.73) node [anchor=north west][inner sep=0.75pt]    {$x$};
% Text Node
\draw (344,25.4) node [anchor=north west][inner sep=0.75pt]    {$y$};
% Text Node
\draw (445.67,25.73) node [anchor=north west][inner sep=0.75pt]    {$x$};
% Text Node
\draw (544.67,26.4) node [anchor=north west][inner sep=0.75pt]    {$y$};
% Text Node
\draw (447,155.4) node [anchor=north west][inner sep=0.75pt]    {$y$};
% Text Node
\draw (543.33,155.07) node [anchor=north west][inner sep=0.75pt]    {$x$};

\end{tikzpicture}
\caption{New Labeling rules.} \label{fig:new_labeling_virtual_crossing_rules}

\end{figure}
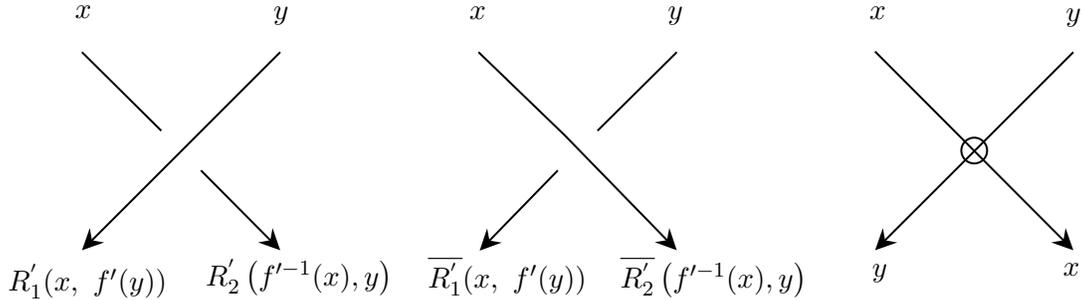

\begin{theorem}\label{thm:virtual_biquandle_isomorphism}
The virtual biquandle $(X,f',R')$ is isomorphic to the fundamental virtual biquandle $(VBQ(L), f, R)$.
\end{theorem}
\begin{proof}
The proof follows from Theorem \ref{thm:both_virtual_biquandles_are_same} with the observation that $(X, f', R')$ is isomorphic to the virtual biquandle $(VBQ(\Psi, \beta),f, R)$, where $\beta \in VB_n$ such that $\cl(\beta)$ is $L$.
\end{proof}

Hence, the fundamental virtual biquandle can be retrieved from just classical crossings in a virtual link.

A {\it Gauss diagram} consists of a finite number of circles oriented anticlockwise with a finite number of
signed arrows whose heads and tails lie on the circles. For every oriented virtual link diagram $L$ one can construct a Gauss diagram, where the arrows corresponds to the classical crossings in $L$. Figure \ref{fig:depiction_of_labeling_rules_on_Gauss_diagrms} depicts the encoding of new labeling rules (from Figure \ref{fig:new_labeling_virtual_crossing_rules}) on Gauss diagrams. Thus, we can recover the fundamental virtual biquandle for a given virtual link $L$ directly from a Gauss diagram representing $L$.

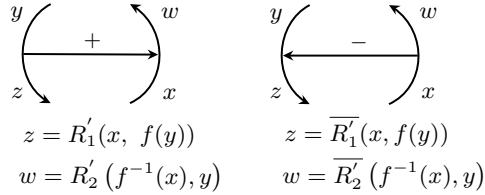
\begin{figure}[H]

\tikzset{every picture/.style={line width=0.75pt}} %set default line width to 0.75pt        

\begin{tikzpicture}[x=0.75pt,y=0.75pt,yscale=-1,xscale=1]

%Curve Lines [id:da5803351564655536] 
\draw    (100.93,51) .. controls (83.42,60.23) and (81.72,88.62) .. (98.44,99.57) ;
\draw [shift={(100.93,101)}, rotate = 206.57] [fill={rgb, 255:red, 0; green, 0; blue, 0 }  ][line width=0.08]  [draw opacity=0] (5.36,-2.57) -- (0,0) -- (5.36,2.57) -- (3.56,0) -- cycle    ;
%Curve Lines [id:da9623691879027874] 
\draw    (143.51,52.64) .. controls (160.86,64.96) and (159.71,92.41) .. (140.93,101) ;
\draw [shift={(140.93,51)}, rotate = 29.54] [fill={rgb, 255:red, 0; green, 0; blue, 0 }  ][line width=0.08]  [draw opacity=0] (5.36,-2.57) -- (0,0) -- (5.36,2.57) -- (3.56,0) -- cycle    ;
%Curve Lines [id:da05204359904277622] 
\draw    (231.56,51.5) .. controls (214.05,60.73) and (212.35,89.12) .. (229.07,100.07) ;
\draw [shift={(231.56,101.5)}, rotate = 206.57] [fill={rgb, 255:red, 0; green, 0; blue, 0 }  ][line width=0.08]  [draw opacity=0] (5.36,-2.57) -- (0,0) -- (5.36,2.57) -- (3.56,0) -- cycle    ;
%Curve Lines [id:da9699555228612915] 
\draw    (274.13,53.14) .. controls (291.49,65.46) and (290.34,92.91) .. (271.56,101.5) ;
\draw [shift={(271.56,51.5)}, rotate = 29.54] [fill={rgb, 255:red, 0; green, 0; blue, 0 }  ][line width=0.08]  [draw opacity=0] (5.36,-2.57) -- (0,0) -- (5.36,2.57) -- (3.56,0) -- cycle    ;
%Straight Lines [id:da627993942475921] 
\draw    (87.13,76.2) -- (152.13,76.49) ;
\draw [shift={(155.13,76.5)}, rotate = 180.25] [fill={rgb, 255:red, 0; green, 0; blue, 0 }  ][line width=0.08]  [draw opacity=0] (5.36,-2.57) -- (0,0) -- (5.36,2.57) -- (3.56,0) -- cycle    ;
%Straight Lines [id:da24491854810339098] 
\draw    (286.03,77) -- (221.03,77) ;
\draw [shift={(218.03,77)}, rotate = 360] [fill={rgb, 255:red, 0; green, 0; blue, 0 }  ][line width=0.08]  [draw opacity=0] (5.36,-2.57) -- (0,0) -- (5.36,2.57) -- (3.56,0) -- cycle    ;

% Text Node
\draw (116.13,64.4) node [anchor=north west][inner sep=0.75pt]  [font=\scriptsize]  {$+$};
% Text Node
\draw (250.43,65.3) node [anchor=north west][inner sep=0.75pt]  [font=\scriptsize]  {$-$};
% Text Node
\draw (79.13,51) node [anchor=north west][inner sep=0.75pt]  [font=\footnotesize]  {$y$};
% Text Node
\draw (210.03,52.1) node [anchor=north west][inner sep=0.75pt]  [font=\footnotesize]  {$y$};
% Text Node
\draw (79.53,91.8) node [anchor=north west][inner sep=0.75pt]  [font=\footnotesize]  {$z$};
% Text Node
\draw (210.03,91.9) node [anchor=north west][inner sep=0.75pt]  [font=\footnotesize]  {$z$};
% Text Node
\draw (155.93,92) node [anchor=north west][inner sep=0.75pt]  [font=\footnotesize]  {$x$};
% Text Node
\draw (286.43,91.7) node [anchor=north west][inner sep=0.75pt]  [font=\footnotesize]  {$x$};
% Text Node
\draw (155.13,51.4) node [anchor=north west][inner sep=0.75pt]  [font=\footnotesize]  {$w$};
% Text Node
\draw (285.43,51.5) node [anchor=north west][inner sep=0.75pt]  [font=\footnotesize]  {$w$};
% Text Node
\draw (85.43,106.9) node [anchor=north west][inner sep=0.75pt]  [font=\footnotesize]  {$z=R_{1}^{'}( x,\ f( y))$};
% Text Node
\draw (83.43,127.9) node [anchor=north west][inner sep=0.75pt]  [font=\footnotesize]  {$w=R_{2}^{'}\left( f^{-1}( x) ,y\right)$};
% Text Node
\draw (217.43,105.9) node [anchor=north west][inner sep=0.75pt]  [font=\footnotesize]  {$z=\overline{R_{1}^{'}}( x,f( y))$};
% Text Node
\draw (216.43,126.9) node [anchor=north west][inner sep=0.75pt]  [font=\footnotesize]  {$w=\overline{R_{2}^{'}}\left( f^{-1}( x) ,y\right)$};

\end{tikzpicture}
\caption{Depiction of new labeling rules on Gauss diagrams.}
\label{fig:depiction_of_labeling_rules_on_Gauss_diagrms}
\end{figure}
 
\section{Concluding remarks}\label{sec:conclusion}
In this article, using virtual braid group representations, we proved that the coloring number invariant of virtual links, derived from virtual biquandles, can be obtained from colorings by biquandles. Moreover, Theorem \ref{thm:virtual_biquandle_isomorphism} suggests that any invariant stemming from the fundamental virtual biquandle can be extracted solely from classical crossings in virtual link diagrams. This suggests that classical crossings in virtual links may encompass more information than previously acknowledged. In our ongoing work \cite{2024arXiv240101533E}, we are integrating the unary operator of virtual biquandles with Alexander numbering of links and standard biquandle coloring, thereby enhancing the state-sum invariant. In our forthcoming research, we aim to delve further into the ``virtualized" invariants, such as the Alexander polynomial as defined in \cite{MR4284604, MR1280610, MR1822148}, and the state-sum invariant derived from virtual biquandles \cite{MR3982044, MR2515811}, shedding further light on their intricacies.

\section*{Acknowledgment}
ME is partially supported by Simons Foundation collaboration grant 712462.  MS is supported by the Fulbright-Nehru postdoctoral fellowship.

\bibliography{references1}{}

\begin{thebibliography}{10}

\bibitem{MR2493369}
V.~G. Bardakov and P.~Bellingeri.
\newblock Combinatorial properties of virtual braids.
\newblock {\em Topology Appl.}, 156(6):1071--1082, 2009.

\bibitem{MR3597250}
V.~G. Bardakov, Y.~A. Mikhalchishina, and M.~V. Neshchadim.
\newblock Representations of virtual braids by automorphisms and virtual knot
  groups.
\newblock {\em J. Knot Theory Ramifications}, 26(1):1750003, 17, 2017.

\bibitem{MR4352636}
V.~G. Bardakov, M.~V. Neshchadim, and M.~Singh.
\newblock Virtually symmetric representations and marked {G}auss diagrams.
\newblock {\em Topology Appl.}, 306:Paper No. 107936, 24, 2022.

\bibitem{MR4284604}
H.~U. Boden and M.~Chrisman.
\newblock Virtual concordance and the generalized {A}lexander polynomial.
\newblock {\em J. Knot Theory Ramifications}, 30(5):Paper No. 2150030, 35,
  2021.

\bibitem{MR3640614}
H.~U. Boden, R.~Gaudreau, E.~Harper, A.~J. Nicas, and L.~White.
\newblock Virtual knot groups and almost classical knots.
\newblock {\em Fund. Math.}, 238(2):101--142, 2017.

\bibitem{MR3717970}
H.~U. Boden and M.~Nagel.
\newblock Concordance group of virtual knots.
\newblock {\em Proc. Amer. Math. Soc.}, 145(12):5451--5461, 2017.

\bibitem{MR1905687}
J.~S. Carter, S.~Kamada, and M.~Saito.
\newblock Stable equivalence of knots on surfaces and virtual knot cobordisms.
\newblock volume~11, pages 311--322. 2002.
\newblock Knots 2000 Korea, Vol. 1 (Yongpyong).

\bibitem{MR2515811}
J.~Ceniceros and S.~Nelson.
\newblock Virtual {Y}ang-{B}axter cocycle invariants.
\newblock {\em Trans. Amer. Math. Soc.}, 361(10):5263--5283, 2009.

\bibitem{MR2153118}
H.~A. Dye and L.~H. Kauffman.
\newblock Minimal surface representations of virtual knots and links.
\newblock {\em Algebr. Geom. Topol.}, 5:509--535, 2005.

\bibitem{2024arXiv240101533E}
M.~{Elhamdadi} and M.~{Singh}.
\newblock {Twisted Yang-Baxter sets, cohomology theory, and application to
  knots}.
\newblock {\em arXiv e-prints}, page arXiv:2401.01533, Jan. 2024.

\bibitem{MR3982044}
M.~A. Farinati and J.~Garc\'{\i}a~Galofre.
\newblock Virtual link and knot invariants from non-abelian {Y}ang-{B}axter
  2-cocycle pairs.
\newblock {\em Osaka J. Math.}, 56(3):525--547, 2019.

\bibitem{MR2100870}
R.~Fenn, M.~Jordan-Santana, and L.~Kauffman.
\newblock Biquandles and virtual links.
\newblock {\em Topology Appl.}, 145(1-3):157--175, 2004.

\bibitem{MR3666513}
R.~A. Fenn and A.~Bartholomew.
\newblock Erratum: {B}iquandles of small size and some invariants of virtual
  and welded knots.
\newblock {\em J. Knot Theory Ramifications}, 26(8):1792002, 11, 2017.

\bibitem{MR2384830}
D.~Hrencecin and L.~H. Kauffman.
\newblock Biquandles for virtual knots.
\newblock {\em J. Knot Theory Ramifications}, 16(10):1361--1382, 2007.

\bibitem{MR1280610}
F.~Jaeger, L.~H. Kauffman, and H.~Saleur.
\newblock The {C}onway polynomial in {${\bf R}^3$} and in thickened surfaces: a
  new determinant formulation.
\newblock {\em J. Combin. Theory Ser. B}, 61(2):237--259, 1994.

\bibitem{MR1914297}
N.~Kamada.
\newblock On the {J}ones polynomials of checkerboard colorable virtual links.
\newblock {\em Osaka J. Math.}, 39(2):325--333, 2002.

\bibitem{MR2100692}
N.~Kamada.
\newblock Span of the {J}ones polynomial of an alternating virtual link.
\newblock {\em Algebr. Geom. Topol.}, 4:1083--1101, 2004.

\bibitem{MR1749502}
N.~Kamada and S.~Kamada.
\newblock Abstract link diagrams and virtual knots.
\newblock {\em J. Knot Theory Ramifications}, 9(1):93--106, 2000.

\bibitem{MR2351010}
S.~Kamada.
\newblock Braid presentation of virtual knots and welded knots.
\newblock {\em Osaka J. Math.}, 44(2):441--458, 2007.

\bibitem{MR2371718}
L.~Kauffman and S.~Lambropoulou.
\newblock The {$L$}-move and virtual braids.
\newblock In {\em Intelligence of low dimensional topology 2006}, volume~40 of
  {\em Ser. Knots Everything}, pages 133--142. World Sci. Publ., Hackensack,
  NJ, 2007.

\bibitem{MR1721925}
L.~H. Kauffman.
\newblock Virtual knot theory.
\newblock {\em European J. Combin.}, 20(7):663--690, 1999.

\bibitem{MR2128049}
L.~H. Kauffman and S.~Lambropoulou.
\newblock Virtual braids.
\newblock {\em Fund. Math.}, 184:159--186, 2004.

\bibitem{MR2191942}
L.~H. Kauffman and V.~O. Manturov.
\newblock Virtual biquandles.
\newblock {\em Fund. Math.}, 188:103--146, 2005.

\bibitem{MR1997331}
G.~Kuperberg.
\newblock What is a virtual link?
\newblock {\em Algebr. Geom. Topol.}, 3:587--591, 2003.

\bibitem{MR1916950}
V.~O. Manturov.
\newblock On invariants of virtual links.
\newblock {\em Acta Appl. Math.}, 72(3):295--309, 2002.

\bibitem{MR1822148}
D.~S. Silver and S.~G. Williams.
\newblock Alexander groups and virtual links.
\newblock {\em J. Knot Theory Ramifications}, 10(1):151--160, 2001.

\bibitem{MR1167178}
M.~Wada.
\newblock Group invariants of links.
\newblock {\em Topology}, 31(2):399--406, 1992.

\end{thebibliography}
\bibliographystyle{abbrv}
\end{document}